\documentclass[11pt,letterpaper]{amsart}

\usepackage[T1]{fontenc}
\usepackage{lmodern}
\usepackage[expansion=false]{microtype}
\usepackage{amsmath,amssymb,mathtools,mathrsfs}
\usepackage[margin=1.08in]{geometry}
\usepackage{xcolor}
\usepackage{xurl}
\usepackage[colorlinks=true,linkcolor=blue,citecolor=red,urlcolor=blue]{hyperref}
\usepackage[alphabetic]{amsrefs}

\makeatletter
\renewcommand{\BibLabel}{%
  \Hy@raisedlink{\hyper@anchorstart{cite.\CurrentBib}\hyper@anchorend}%
  [\thebib]%
}
\makeatother
\renewcommand{\PrintDOI}[1]{\href{https://doi.org/#1}{DOI \nolinkurl{#1}}}

\numberwithin{equation}{section}
\theoremstyle{plain}
\newtheorem{theorem}{Theorem}[section]
\newtheorem{maintheorem}{Theorem}

\newtheorem{proposition}[theorem]{Proposition}
\newtheorem{lemma}[theorem]{Lemma}
\theoremstyle{definition}
\newtheorem{example}[theorem]{Example}
\theoremstyle{remark}
\newtheorem{remark}[theorem]{Remark}

\allowdisplaybreaks[1]
\makeatletter
\newcommand{\proofstep}[2]{%
  \if@inlabel\leavevmode\par\nobreak\else\par\fi
  \addvspace{.35\baselineskip}%
  \noindent\textbf{Step~#1. #2.}\nobreak\enspace\ignorespaces}
\makeatother
\newcommand{\C}{\mathbb C}

\newcommand{\OO}{\mathcal O}
\newcommand{\rk}{\operatorname{rank}}
\newcommand{\Levi}{\operatorname{Levi}}
\newcommand{\supp}{\operatorname{supp}}
\newcommand{\Rea}{\operatorname{Re}}

\newcommand{\dd}{\mathrm d}
\newcommand{\db}{\bar\partial}
\newcommand{\restr}{\operatorname{res}}
\newcommand{\cR}{\mathcal R}
\newcommand{\crk}{\operatorname{crk}}
\newcommand{\rescount}{N}
\newcommand{\inmap}{\mathsf A}
\newcommand{\outmap}{\mathsf B}

\hypersetup{
 pdftitle={Curvature rank, numerical dimension, and cohomology growth},
 pdfauthor={Xiangsen Qin},
 pdfsubject={Cohomology of coherent analytic sheaves and nef line bundles},
 pdfkeywords={Curvature rank, cohomology of coherent analytic sheaves, nef line bundles, multiplier ideal sheaves}
}

\begin{document}

\title[Cohomology Growth]
{Curvature rank, numerical dimension, and cohomology growth}
\author[X. Qin]{Xiangsen Qin}
\address{Xiangsen Qin: Chern Institute of Mathematics and LPMC, Nankai University \\
Tianjin 300071, China}
\email{\href{mailto:qinxiangsen@nankai.edu.cn}{qinxiangsen@nankai.edu.cn}}
\date{}

\begin{abstract}
Let $L$ be a holomorphic line bundle on a compact complex manifold $X$ of
dimension $n$, and let $r$ be the least integer such that the first
Bott--Chern class of $L$ is represented by a smooth closed real
$(1,1)$-form of rank at most $r$ at every point. 
For any nonzero coherent analytic sheaf $\mathcal F$ on $X$
whose support has dimension $d$, we prove that
$h^q(X,\mathcal F\otimes L^p)=O_\varepsilon(p^{r+\varepsilon})$ for every
$\varepsilon>0$ and every $0\le q\le d$. This is the
exponent conjectured by Ma--Marinescu--Zelditch, up to an arbitrarily
small loss. In the extreme degrees $q=0$ and $q=d,$  the
bound can be taken to be  $O(p^{\min\{r,d\}})$. If, moreover, $X$ is in Fujiki's class and $L$ is nef, then
$h^q(X,\mathcal F\otimes L^p)=O(p^{\min\{\nu(L),\,d-q\}})$, where $\nu(L)$
is the numerical dimension of $L$. This answers a question of Demailly
for compact Fujiki manifolds. Finally, for arbitrary $X$, any singular Hermitian metric \(h\) on \(L\) with semipositive curvature  and analytic singularities,
and any  holomorphic vector bundle \(E\) on $X$, we have 
\(h^q(X,E\otimes L^p\otimes\mathcal I(h^{\otimes p}))=O_\varepsilon(p^{\min\{n-q,b+\varepsilon\}})\)
for any  $0\leq q\leq n$, where  \(b\) is the maximum curvature rank outside the pole set of $h$,
and  $\mathcal I(h^{\otimes p})$ is the multiplier ideal sheaf. When $X$ belongs to Fujiki's class, the exponent is \(\min\{n-q,b\}\).
\end{abstract}


\maketitle
\setcounter{tocdepth}{1}
\tableofcontents

\section{Introduction}\label{sec:introduction}

Let \(L\) be a holomorphic line bundle on a compact complex manifold \(X\)
of dimension \(n\), and put \(L^p=L^{\otimes p}\). 
For a coherent analytic sheaf \(\mathcal G\) on \(X\), let
\(H^q(X,\mathcal G)\) denote its \(q\)-th sheaf cohomology group
and write \(h^q(X,\mathcal G)=\dim_{\mathbb C}H^q(X,\mathcal G)\).
For any smooth Hermitian metric \(h\) on \(L\), let \(\Theta_{L,h}\) be
\(i\) times its Chern curvature $R^L$, and let \(X(q)\) denote the open subset
of \(X\) where \(\Theta_{L,h}\) is nondegenerate with exactly \(q\)
negative eigenvalues. As \(p\to\infty\),
Demailly's weak holomorphic Morse inequalities \cite{Demailly1985}*{Th\'eor\`eme~0.1(a)}
give
\[
 h^q(X,L^p)\le \frac{p^n}{n!}\int_{X(q)}(-1)^q \left(\frac{\Theta_{L,h}}{2\pi}\right)^n+o(p^n)
 \qquad (0\le q\le n).
\]
 The leading term
sees only the points where the curvature is nondegenerate.
If $\Theta_{L,h}$ is everywhere degenerate, each \(X(q)\) is empty and the
inequalities yield only \(h^q(X,L^p)=o(p^n)\). What power bound for
\(h^q(X,L^p)\) follows from a pointwise bound on the curvature rank?

Let \(r(h)\) be the maximum complex rank of
\(\Theta_{L,h}\) on \(X\). For any fixed holomorphic vector bundle
\(E\) on \(X\), Ma--Marinescu--Zelditch \cite{MMZ2015}*{\S1.2, p.~181, after (1.30)}
conjecture the following. If \(r(h)\le n-s\) for an integer \(1\le s\le n\), then for
any nondecreasing positive function $a(p)\to\infty$,
\[
 h^q(X,L^p\otimes E)\le a(p)^s p^{n-s}
 \qquad(0\le q\le n)
\]
for all integers $p\ge p_0$, where $p_0$ may depend on $a$, $s$,
and the fixed geometric data. For fixed \(s\), the quantifier over \(a\)
makes this equivalent to \(h^q(X,L^p\otimes E)=O(p^{n-s})\) for every
\(q\) (Remark~\ref{rem:slow-factor}).

With \(X,L,E\) fixed, the dimensions \(h^q(X,L^p\otimes E)\) do not
depend on \(h\), whereas the rank of \(\Theta_{L,h}\) does. To remove this
dependence, let \(\crk(L)\) be the least integer \(r\) such that the first Bott--Chern class
\(c_1^{\mathrm{BC}}(L)\) is represented by a smooth closed real
\((1,1)\)-form of rank at most \(r\) at every point of \(X\). We call
\(\crk(L)\) the \emph{curvature rank} of \(L\). 
If \(\crk(L)<n\), the following theorem improves on the \(o(p^n)\)
estimate from Demailly's weak holomorphic Morse inequalities
\cite{Demailly1985}*{Th\'eor\`eme~0.1(a)}. 
\begin{maintheorem}\label{thm:main}\label{fi:thm:iteration}
Let \(r=\crk(L)\), and let \(\mathcal F\) be a nonzero coherent
analytic sheaf on \(X\) with support of dimension \(d\). For every \(\varepsilon>0\),
\begin{equation}\label{eq:near-endpoint}
 h^q(X,\mathcal F\otimes L^p)=O_\varepsilon(p^{r+\varepsilon})
 \qquad(0\le q\le d).
\end{equation}
Moreover, when $q=0,d$, the exponent can be taken to be \(\min\{r,d\}\).
\end{maintheorem}
For intermediate curvature rank, \eqref{fi:eq:iteration} gives a better logarithmic estimate.
For \(q>d\), the $q$-th cohomology group vanishes \cite{RamisRuget1970}*{p.~89}.
For \(0\le q\le d\), Demailly's estimate gives \(h^q(X,\mathcal F\otimes L^p)=O(p^d)\)
\cite{Demailly2010}*{Lemma~2.1}.

When \(\mathcal F\) is a holomorphic vector bundle and \(r<n\),
Theorem~\ref{thm:main} gives the exponent \(r\) predicted by
Ma--Marinescu--Zelditch in degrees \(0\) and \(n\). In intermediate
degrees, it gives the exponent \(r+\varepsilon\) for every
\(\varepsilon>0\). 
Bouche obtained the bound \(O(p^r)\) in all degrees when \(X\) admits a smooth
foliation of complex codimension \(r\) whose tangent bundle is pointwise contained in the curvature kernel
\cite{Bouche1991}*{Corollaire~0.2}. Z.~Wang retained this foliation hypothesis for local Bergman kernel
inequalities with two powers \(k,l\) tending to infinity and \(k/l\to\infty\)
\cite{ZWangLocal}*{Theorems~1.1--1.3}.
If the curvature form has constant rank \(r\), its kernel defines
such a foliation. However, when the rank of the curvature form varies, such a foliation may
fail to exist, as Example~\ref{ex:rank-jump} shows.

Can one improve the bound in Theorem~\ref{thm:main} under an
additional positivity assumption on \(L\)?
If \(L\) admits a smooth
Hermitian metric with semipositive curvature, then for every fixed
holomorphic vector bundle \(E\) on \(X\),
\[
 h^q(X,E\otimes L^p)=O(p^{n-q}),\qquad q\ge1.
\]
Apply the spectral estimate for \((n,q)\)-forms at eigenvalue zero,
with \(K_X^{-1}\otimes E\) as the coefficient bundle. When \(E\) is a line
bundle, this is Berndtsson's estimate \cite{Berndtsson2002}*{Theorem~1.1}.
Wang and Zhou give the vector bundle version \cite{WangZhou2019}*{Theorem~1.3}.
It is natural to ask whether the same degree bound holds under the
weaker assumption that \(L\) is nef.

Fix a Hermitian form \(\omega\) on \(X\). A holomorphic line bundle \(L\)
is nef if for every \(\delta>0\) it admits a smooth Hermitian metric
\(h_\delta\) with \(\Theta_{L,h_\delta}\ge-\delta\omega\)
\cite{DPS1994}*{Definition~1.2}. A nef line bundle need
not admit a smooth metric with semipositive curvature \cite{DPS1994}*{Example~1.7}.
Write \(c_1(L)\in H^2(X,\mathbb R)\) for the  first Chern class of \(L\).
For nef \(L\), we set
\[
 \nu(L)=\max\{j\ge0:c_1(L)^j\ne0\text{ in }H^{2j}(X,\mathbb R)\}.
\]
On a compact K\"ahler manifold, this is the usual numerical dimension
\cite{DemaillyCetraro}*{Definition~(6.9)}.
It is clear that $\nu(L)\leq \crk(L)$.

Recall that a compact complex manifold belongs to Fujiki's class \(\mathcal C\)
precisely when it is bimeromorphic to a compact K\"ahler manifold
\cite{Varouchas1989}*{Ch.~IV, Theorem~5}.

\begin{maintheorem}\label{thm:nef}
Let \(L\) be a nef line bundle on a Fujiki manifold \(X\), and let \(\mathcal F\) be a fixed nonzero
coherent analytic sheaf on \(X\) with support of dimension \(d\). Then for any $0\leq q\leq d$,
\begin{equation}\label{eq:nef-main}
 h^q(X,\mathcal F\otimes L^p)=O\!\left(p^{\min\{\nu(L),d-q\}}\right).
\end{equation}
\end{maintheorem}

When \(X\) is projective, Theorem~\ref{thm:nef} recovers Fujita's
estimate \cite{Fujita1980}, \cite{FujitaSemipositive}*{Theorems~(3.3) and~(6.2)}.
For nef \(L\) and a fixed holomorphic vector bundle \(E\), Demailly
obtained \(o(p^n)\) in positive degrees on compact K\"ahler manifolds
and \(O(p^{n-q})\) on projective manifolds
\cite{DemaillyCetraro}*{Lemma~(6.7)}. He asked whether the latter
bound holds without projectivity
\cite{DemaillyCetraro}*{end of the proof of Lemma~(6.7)}.
He later noted that the question remained open even for compact
K\"ahler manifolds \cite{DemaillyCIME}*{Note~22, after (42), p.~169}.
Theorem~\ref{thm:nef} answers this question on manifolds in Fujiki's
class.

Now, we consider the corresponding estimates for singular Hermitian
metrics on an arbitrary compact complex manifold.
In a local holomorphic frame \(e\), a singular Hermitian metric \(h\) on \(L\)
has the form \(|e|_h^2=e^{-\phi}\) with \(\phi\) locally integrable.
$(L,h)$ is semipositive if \(i\partial\bar\partial\phi\ge0\)
in the sense of currents. The metric has analytic singularities if locally
\[
 \phi=c\log\sum_{\ell=1}^N|f_\ell|^2+\psi ,
\]
where \(c>0\), the \(f_\ell\) are holomorphic functions, and \(\psi\)
is smooth (compare \cite{WangZhou2019}*{Definition~2.1}).
The pole set \(Z_h\) of \(h\) is locally given by \(\{\phi=-\infty\}\). The multiplier
ideal \(\mathcal I(h^{\otimes p})\) consists of the germs \(g\) for
which \(|g|^2e^{-p\phi}\) is locally integrable. These definitions
are independent of the local frame. If the local coefficients \(c\) can
be chosen rational, the singularities are called algebraic.

For metrics with algebraic singularities, Bonavero's singular holomorphic
Morse inequalities give an upper bound
for \(h^q(X,L^p\otimes\mathcal I(h^{\otimes p}))\) in terms of
the curvature integral over the index set in \(X\setminus Z_h\)
\cite{Bonavero1998}, \cite{MaMarinescu2007}*{Theorem~2.3.18}.
When $h$ has semipositive curvature, the index set is empty in
positive degrees, so the bound reduces to \(o(p^n)\).
When \(X\) is a compact K\"ahler manifold and \(h\) has semipositive curvature, Matsumura \cite{MatsumuraQuestion}*{Problem~3.4} asked whether,
for every fixed holomorphic vector bundle \(E\) on \(X\), one has
\[
 h^q\bigl(X,E\otimes L^p\otimes\mathcal I(h^{\otimes p})\bigr)
 =O(p^{n-q})
  \qquad\text{as }p\to\infty.
\]
Matsumura proved the estimate for smooth projective \(X\)
\cite{Matsumura2014}*{Theorem~4.1}. Wang and Zhou posed the question
on compact Hermitian manifolds and proved the estimate for semipositive
metrics with algebraic singularities 
\cite{WangZhou2019}*{Question~1.2 and Theorem~1.4}. When \(E\) is a
line bundle and \(q\ge1\), see also \cite{HuanWang2021}*{Theorem~1.7}.
The following theorem allows positive irrational coefficients in the local weights
and also takes account of the curvature rank outside the pole set.

\begin{maintheorem}\label{thm:singular}\label{prop:pole-estimates}
Let \(E\) be a fixed holomorphic vector bundle on a compact Hermitian
manifold \(X\) of dimension \(n\), and let \(h\) be a singular Hermitian
metric on the holomorphic line bundle \(L\) with semipositive curvature and analytic
singularities. For every \(0\le q\le n\),
\begin{equation}\label{eq:singular-main}
 h^q\bigl(X,E\otimes L^p\otimes\mathcal I(h^{\otimes p})\bigr)
       =O(p^{n-q}).
\end{equation}
For every \(\varepsilon>0\), these dimensions are also
bounded by \(O_\varepsilon(p^{b+\varepsilon})\), where \(b\) is the maximum rank of \(\Theta_{L,h}\) on
\(X\setminus Z_h\). If \(b=0\), they are bounded.
If \(X\in\mathcal C\), the bound can be taken to be \(O(p^{\min\{n-q,b\}})\).
\end{maintheorem}
A logarithmic refinement of the curvature rank estimate is given
in \eqref{eq:rounded-log}.

For \(q>\dim Z_h\), all the bounds in Theorem~\ref{thm:singular}
also hold for \(h^q(X,E\otimes L^p)\) by \eqref{eq:above-pole-surjection}.
Wang--Zhou obtained the degree bound in this range for algebraic
singularities and, under an additional curvature hypothesis, for analytic
singularities \cite{WangZhou2019}*{Theorems~1.5--1.6}.

For a coherent analytic sheaf \(\mathcal F\) with support of
dimension \(d\), the estimate
\[
 h^q\bigl(X,\mathcal F\otimes L^p\otimes\mathcal I(h^{\otimes p})\bigr)
 =O(p^{d-q})
\]
can fail even for \(q=0\) (Example~\ref{ex:point-multiplier}).

\subsection*{Outline of the proofs}

Lemma~\ref{lem:support-reduction} bounds cohomology with coherent
coefficients by a finite sum of cohomology dimensions with vector
bundle coefficients on compact complex manifolds mapping to \(X\).
We therefore first prove Theorem~\ref{thm:main} for a vector bundle \(E\).
Choose a smooth Hermitian metric \(h\) such that \(r(h)=r=\crk(L)\). Composing local
Dolbeault homotopies gives a global identity, up to chain homotopy,
in terms of holomorphic restrictions between nested domains
(Lemma~\ref{fi:lem:composition} and
Proposition~\ref{fi:it:prop:assembly}). On the final domains, a local
holomorphic change of frame makes the weight of the pullback of
\(L^p\) uniformly bounded. After harmonic projection the identity factors through the direct
sum of the holomorphic restrictions. Their Taylor approximants of
logarithmic degree give an injective map on harmonic forms by
\eqref{eq:finite-rank-principle}. Counting their ranks gives
\eqref{fi:eq:localization}.

For a Hermitian quadratic weight of rank \(r\), only \(r\) complex
directions need to be subdivided at scale \(p^{-1/2}\). For a general
smooth local weight \(\phi\), a change of frame removes the
pluriharmonic terms of its quadratic Taylor polynomial, leaving a
Hermitian form and an \(O(t^3)\) remainder on a box of width \(t\).
At centers of locally constant rank, parabolic rescaling gives a
fourth-order term with a bounded holomorphic factor and an \(O(t^5)\) remainder,
uniformly as the nonzero curvature eigenvalues tend to zero
(Lemma~\ref{fi:oa:parabolic}).

For \(0<r<n\), fix the number \(k\) of further rescalings and retain
the Taylor polynomial of \(\phi\) through degree \(2k+4\). The product
charts of Proposition~\ref{fi:cell:prop:targetproducts} allow the
rescaling to be repeated without inverting the leading holomorphic
factor. Restoring the scale factors gives a remainder of order
\(t^{2k+5}\) in the original potential
\eqref{fi:it:eq:final-expansion}. We take \(pt^{2k+5}=1\).
Up to logarithmic factors, the restriction count is then bounded by
a constant times
\[
 t^{-2n}(pt^2)^r=p^{\,r+2(n-r)/(2k+5)}
\]
by \eqref{fi:it:eq:global}. Given \(\varepsilon>0\), we choose \(k\)
with \(2(n-r)/(2k+5)<\varepsilon\) before letting \(p\to\infty\).
The logarithmic factor in \eqref{fi:eq:iteration} is then absorbed
in the \(O_\varepsilon(p^{r+\varepsilon})\) bound.

For Theorem~\ref{thm:nef}, we first take vector bundle coefficients
on a compact K\"ahler manifold \((X,\omega)\). Nefness gives metrics
on \(L^p\) with curvature bounded below by \(-\omega\). We introduce
an auxiliary smooth line bundle \(A_m\) with a unitary connection
whose curvature satisfies \(iR^{A_m}=m\omega+O(1)\). For the local mass estimate \eqref{nef:local-mass} we use
Berndtsson's argument \cite{Berndtsson2002}*{proof of Theorem~2.1} with
the vector bundle formula of Wang and Zhou \cite{WangZhou2019}*{Proposition~3.1}.
Combining this estimate with the spectral gap of the twisted Dirac
operator gives the trace bound \eqref{nef:trace-index}.
For fixed sufficiently large \(m\), the index is \(O(p^{\nu(L)})\).
For \(m=p\), the index is \(O(p^n)\) and the local mass estimate
contributes \(p^{-q}\), giving \(O(p^{n-q})\). Applying both bounds on the compact K\"ahler manifolds in
Lemma~\ref{lem:support-reduction} gives \eqref{eq:nef-main}.

For Theorem~\ref{thm:singular}, we round down the coefficients of the
divisor on a log resolution. The Leray spectral sequence expresses the
groups in \eqref{eq:singular-main} as cohomology groups on the resolved
manifold with a fixed vector bundle twist (Section~\ref{sec:analytic-poles}).
Wang--Zhou group the powers \(p\) by residue class in the rational
case \cite{WangZhou2019}*{\S6, (15)--(16)}. For real coefficients the
fractional parts give smooth corrections to the weights, bounded
independently of \(p\). Demailly and Rahmati noted the same need for
uniform estimates in their Morse inequalities
\cite{DemaillyRahmati2019}*{Remark~4.8}. We obtain the degree bound \eqref{eq:rounded-degree}
by adapting Berndtsson's rescaling argument
\cite{Berndtsson2002}*{proof of Theorem~2.3} to the resulting weights. For the rank estimate we use the same homotopy on a fixed
cover, now with the weights from the resolution. Semipositivity of
the residual curvature gives uniform local solution bounds. We can
therefore use the Hilbert--Schmidt estimate \eqref{eq:localization-hs}.
On each fixed ball, point extension and comparison of plurisubharmonic
weights give the logarithmic bound \eqref{eq:rounded-log}. In Fujiki's
class the K\"ahler resolution gives the sharper rank bound by
\eqref{nef:trace-index}.

\section{Preliminaries}\label{sec:conventions}

\subsection{Notation and Conventions}
Throughout, \(X\) is a compact complex manifold of dimension \(n\).
All dimensions are complex. We often identify a holomorphic vector
bundle with its sheaf of holomorphic sections.  For open subsets \(U\) and \(V\) of a topological space, write \(U\Subset V\) if the closure of
\(U\) is compact and contained in \(V\).

For a complex manifold \(M\) of dimension $m$, we use $\dim M=m$ to denote its dimension.
Let \(\OO_M\) be its sheaf of
holomorphic functions. For a holomorphic line bundle \(L\) on \(M\), let \(L^*\) be its dual.
For a positive integer \(p\), put \(L^p=L^{\otimes p}\) and
\(L^{-p}=(L^*)^{\otimes p}\), and set \(L^0=\OO_M\). For \(0\le j\le m\),
write \(\Omega_M^j\) for the bundle of holomorphic \(j\)-forms on \(M\).
The canonical bundle is \(K_M=\Omega_M^m\).

Let \(\nabla^V\) be a unitary connection with curvature \(R^V=(\nabla^V)^2\) on a
Hermitian bundle \(V\). For a Hermitian holomorphic line bundle \((L,h)\), we set \(\Theta_{L,h}=iR^L\), where \(R^L\) is the Chern curvature of $(L,h)$.
 In a holomorphic frame \(e\) with \(|e|_h^2=e^{-\phi}\),
\[
 \nabla^L e=-\partial\phi\otimes e,\qquad
 \Theta_{L,h}=i\partial\bar\partial\phi.
\]
When the metric is clear, we often omit the subscript $h$.
Here \(\partial,\bar\partial\) are the two type components of exterior
differentiation. Let \(c_1(L)\) and \(c_1^{\mathrm{BC}}(L)\) denote the first
Chern class of $L$ and the first Bott--Chern class of \(L\), respectively.
The form \(c_1(L,h):=\Theta_{L,h}/(2\pi)\) is a smooth representative
of both \(c_1(L)\) and \(c_1^{\mathrm{BC}}(L)\)
\cite{MaMarinescu2007}*{Theorem~1.1.5 and Definition~2.3.2}.
Let \(f\colon Y\to M\) be a holomorphic map from a complex manifold
\(Y\). We equip \(L^p\), \(L^*\), and the pullback line bundle \(f^*L\)
with the induced Hermitian metrics. For a real \((1,1)\)-form \(\alpha\),
let \(\rk_x\alpha\) denote its Hermitian rank at \(x\). Then, for any \(y\in Y\),
\begin{equation}\label{eq:pullback-rank}
 \Theta_{L^p}=p\Theta_L,\qquad \Theta_{L^*}=-\Theta_L,\qquad
 \rk_y\Theta_{f^*L}\le \rk_{f(y)}\Theta_L.
\end{equation}

For a closed analytic subset $W$ of a complex manifold \(M\), the structure sheaf of $W$ is given by \(\OO_W=(\OO_M/\mathcal I_W)|_W\),
where \(\mathcal I_W\subset\OO_M\) is the ideal sheaf of \(W\) in \(M\), consisting of germs of holomorphic functions that vanish on \(W\). For a
coherent analytic sheaf \(\mathcal G\) on \(W\), write
\(\operatorname{Supp}\mathcal G\) for the set of points where its
stalk is nonzero. This support is a closed analytic subset of \(W\)
\cite{DemaillyCADG}*{Chap.~II, Theorem~(5.10), p.~105}.
For an integer \(q\), let \(H^q(W,\mathcal G)\)
be its sheaf cohomology and set
\(h^q(W,\mathcal G)=\dim_{\mathbb C}H^q(W,\mathcal G)\). We take both to be zero for \(q<0\). 
For \(q>\dim\operatorname{Supp}\mathcal G\), we have
\(H^q(W,\mathcal G)=0\) by \cite{RamisRuget1970}*{p.~89}, applied
on the closed analytic subspace defined by the annihilator of \(\mathcal G\).
If \(W\) is compact, the Cartan--Serre finiteness theorem gives
\(h^q(W,\mathcal G)<\infty\) for every \(q\)
\cite{CartanSerre1953}, \cite{DemaillyCADG}*{Chap.~IX, Theorem~(4.8), p.~424}.

Let $Y,Z$ be closed analytic subsets of possibly different complex manifolds and let \(\mathcal F\)
and \(\mathcal G\) be coherent analytic sheaves on \(Y\) and \(Z\),
respectively. For a holomorphic map \(f\colon Y\to Z\), write \(f^*\mathcal G\) for the pullback sheaf of $\mathcal G$ and
\(f_*\mathcal F\) for the direct image sheaf of $\mathcal{F}$. For \(j>0\),
\(R^j f_*\mathcal F\) denotes the \(j\)-th higher direct image, and
\(R^0 f_*\mathcal F=f_*\mathcal F\). For any $j\ge0$, $R^j f_*\mathcal F$ is coherent
if \(f\) is proper \cite{Grauert1960}*{\S6, Hauptsatz~I}.

Let $Z$ be a closed analytic subset of a complex manifold $M$, and let \(i\colon Z\hookrightarrow M\) be the inclusion map. For a coherent analytic
sheaf \(\mathcal F\) on \(Z\) and a holomorphic line bundle \(L\) on
\(M\).
We use  $\mathcal T(\mathcal F)$ to denote the torsion sheaf of $\mathcal F$. 
For each open subset \(U\) of $Z$, \(\mathcal T(\mathcal F)(U)\) is given by
\[
 \mathcal T(\mathcal F)(U)
 =\{s\in\mathcal F(U): s|_{U_0}=0
   \text{ for some dense open subset }U_0\subset U\}.
\]
The subsheaf \(\mathcal T(\mathcal F)\) is coherent and its support
is nowhere dense in \(Z\) \cite{AxelssonMagnusson1998}*{\S2.3, p.~13}.
Here a closed subset of \(Z\) is nowhere dense if it has empty interior in \(Z\).
Every coherent subsheaf of \(\mathcal F\) with nowhere dense support
is contained in \(\mathcal T(\mathcal F)\), since its sections vanish
on the dense open complement of its support.
We say that \(\mathcal F\) is torsion-free if
\(\mathcal T(\mathcal F)=0\).

\subsection{Reduction to vector bundles}\label{sec:coherent}
In this subsection, we record the following possibly well known lemma, which alllows us 
to assume the coherent analytic sheaf $\mathcal{F}$ in Theorems~\ref{thm:main}
and~\ref{thm:nef} is a holomorphic vector bundle. We contain a proof of  the following lemma for completeness.
\begin{lemma}\label{lem:support-reduction}
Let \(\mathcal F\) be a nonzero coherent analytic sheaf on \(X\) with support
of dimension \(d\). For each integer \(0\le q\le d\), there are finitely many
compact complex manifolds $Y_j$, holomorphic maps \(f_j:Y_j\to X\),
holomorphic vector bundles \(V_j\) on \(Y_j\), and integers \(0\leq q_j\leq q\) such that
\begin{equation}\label{eq:support-comparison}
 \dim Y_j-q_j\leq d-q,\qquad  h^q(X,\mathcal F\otimes L^p)
 \le \sum_j h^{q_j}\bigl(Y_j,V_j\otimes (f_j^*L)^p\bigr)
 \qquad(p\ge1).
\end{equation}
If \(X\in\mathcal C\), the manifolds \(Y_j\) may be chosen to be K\"ahler.
\end{lemma}

\begin{proof}
If \(d=0\), then \(q=0\) and \(\operatorname{Supp}\mathcal F\)
is a finite set. For each \(x\in\operatorname{Supp}\mathcal F\), let
\(\mathfrak m_x\) be the maximal ideal of \(\mathcal O_{X,x}\). The analytic Nullstellensatz
\cite{DemaillyCADG}*{Chap.~II, (4.22), p.~97}
gives \(\mathfrak m_x^{N_x}\mathcal F_x=0\) for some \(N_x\).
Thus \(\mathcal F_x\) is a finitely generated module over the finite-dimensional
algebra \(\mathcal O_{X,x}/\mathfrak m_x^{N_x}\), so it is finite-dimensional over \(\mathbb C\)
\cite{DemaillyCADG}*{Chap.~IX, \S1.A, p.~404}.
Set $\ell_x:=\dim_{\mathbb{C}}\mathcal F_x$.
The global sections of a sheaf with finite support are the direct sum
of its stalks on that support. Since \(L^p\) is locally free of rank one,
\begin{equation}
h^0(X,\mathcal F\otimes L^p)=\sum_{x\in\operatorname{Supp}\mathcal F}\ell_x
\qquad(p\ge1).
\end{equation}
For each $x\in \operatorname{Supp}\mathcal F$, take \(Y_x=\{x\}\), let \(f_x:Y_x\hookrightarrow X\) be the inclusion,
and set \(V_x=\mathbb C^{\ell_x}\) and \(q_x=0\). Then
\begin{equation}
h^0\bigl(Y_x,V_x\otimes(f_x^*L)^p\bigr)=\ell_x,
\qquad
\dim Y_x-q_x=0=d-q.
\end{equation}
This proves Lemma~\ref{lem:support-reduction} for \(d=0\). Now we only need to treat the case $d>0$.

\proofstep{1}{Reduce to a torsion-free sheaf on a $d$-dimensional irreducible analytic subset of $X$}
By induction on the support dimension, we may assume that
\(\mathcal F\) is a torsion-free sheaf on an irreducible component of
dimension \(d\). To see this, write \(\operatorname{Supp}\mathcal F=\bigcup_{i=1}^t Z_i\)
for its irreducible decomposition and let
\(\mathcal J_i=\mathcal I_{Z_i}\).
The analytic Nullstellensatz
\cite{DemaillyCADG}*{Chap.~II, (4.22), p.~97}
and compactness of $X$ give
\begin{equation}
(\mathcal J_1\cdots\mathcal J_t)^N\mathcal F=0
\end{equation}
for some positive integer \(N\). Let \(m:=Nt\) and set \(\mathcal G_0=\mathcal F\).
For \(a=kt+i\), where \(0\le k<N\) and \(1\le i\le t\), define
\[
\mathcal G_a
=(\mathcal J_1\cdots\mathcal J_t)^k
 \mathcal J_1\cdots\mathcal J_i\,\mathcal F.
\]
Then we get a filtration
\begin{equation}
\mathcal F=\mathcal G_0\supset\mathcal G_1
\supset\cdots\supset\mathcal G_m=0.
\end{equation}
For \(a=kt+i\) as above, the quotient
\(\mathcal Q_a=\mathcal G_{a-1}/\mathcal G_a\) satisfies
\(\mathcal J_i\mathcal Q_a=0\). Let \(\iota_i:Z_i\hookrightarrow X\) be the inclusion. Then
\begin{equation}
\mathcal Q_a\cong(\iota_i)_*E_a,
\qquad
E_a:=\iota_i^*\mathcal Q_a.
\end{equation}
Clearly, $E_a$ is a coherent analytic sheaf on $Z_i$. Since \(L^p\) is locally free, the projection formula \cite{Hartshorne1977}*{Chap.~II, Exercise~5.1(d), p.~124}, the long exact sequences of sheaf cohomology and the
invariance of sheaf cohomology under closed embeddings
\cite{DemaillyCADG}*{Chap.~IV, Corollary~(13.10), p.~234} then give 
\begin{equation}
\begin{aligned}
h^q(X,\mathcal F\otimes L^p)
&\le \sum_{a=1}^m h^q(X,\mathcal Q_a\otimes L^p)=\sum_{k=0}^{N-1}\sum_{i=1}^t
h^q\bigl(Z_i,E_{kt+i}\otimes(L|_{Z_i})^p\bigr).
\end{aligned}
\end{equation}

For \(a=kt+i\) as above, set \(T_a:=\mathcal T(E_a)\) and \(F_a:=E_a/T_a\).
The exact sequence
\begin{equation}
0\longrightarrow T_a\longrightarrow E_a
\longrightarrow F_a\longrightarrow0
\end{equation}
similarly gives
\begin{equation}\label{nozero}
h^q\bigl(Z_i,E_a\otimes(L|_{Z_i})^p\bigr)
\le
h^q\bigl(Z_i,T_a\otimes(L|_{Z_i})^p\bigr)
+
h^q\bigl(Z_i,F_a\otimes(L|_{Z_i})^p\bigr).
\end{equation}
The support of \(T_a\) is a proper analytic subset of \(Z_i\).
Thus \(T_a\), and also \(F_a\) when \(\dim Z_i<d\), have support of
dimension less than \(d\). For each such nonzero sheaf \(H\), put
\(e=\dim\operatorname{Supp}H\). As above,
\begin{equation}
h^q\bigl(Z_i,H\otimes(L|_{Z_i})^p\bigr)
=h^q\bigl(X,(\iota_i)_*H\otimes L^p\bigr).
\end{equation}
This term vanishes for \(q>e\)  by \cite{RamisRuget1970}*{p.~89}. For \(q\le e\leq d\), apply induction to
\((\iota_i)_*H\) on \(X\), omitting the zero summands with \(q_j>\dim Y_j\).
It remains to treat a nonzero
torsion-free sheaf \(F_a\) on a component \(Z=Z_i\) of dimension \(d\).
\proofstep{2}{Make the pullback locally free modulo torsion}
Choose a proper bimeromorphic holomorphic map \(g:Y\to Z\) with \(Y\)
smooth and \(V=g^*F/\mathcal T(g^*F)\) locally free. Apply Rossi's
theorem \cite{Rossi1968}*{Theorem~3.5} and resolve the source
\cite{Wlodarczyk2009}*{Theorem~2.0.1}, as in
\cite{RuppenthalSera2017}*{p.~238 and Theorem~2.3}. Wu uses this
construction for the holomorphic Morse inequalities with coherent
coefficients \cite{XWuCoherent}.
If \(X\in\mathcal C\), then \(Z,Y\in\mathcal C\)
\cite{Fujiki1978}*{Lemma~4.6(2)--(3)}. In that case pull \(V\) back
to a compact K\"ahler manifold mapping properly and bimeromorphically
to \(Y\) \cite{Varouchas1989}*{Ch.~IV, Theorem~5}. Denote the new source,
the composite map to \(Z\), and the pulled-back bundle again by \(Y,g,V\).

The map \(F\to g_*V\) is generically an isomorphism. Its kernel is
zero by torsion-freeness, and its cokernel \(Q\) has support of
dimension at most \(d-1\). For \(b>0\), put
\(T_b=\operatorname{Supp}R^b g_*V\). Every fiber over \(T_b\) has
dimension at least \(b\), by upper semicontinuity of fiber dimension
for proper maps and the vanishing of the higher direct images
above the fiber dimension \cite{BarletVarouchas1989}*{p.~330, Corollaire~(i)}.
If \(T_b\ne\varnothing\), the fiber dimension theorem gives
\[
 \dim T_b+b\le\dim g^{-1}(T_b)\le d-1 .
\]
The last inequality holds since \(g\) is bimeromorphic.

\proofstep{3}{Apply the Leray spectral sequence}
Set \(M=L|_Z\). By the projection formula, the Leray spectral sequence
\cite{Leray1950}*{Chap.~II, \S VI, no.~50},
\cite{DemaillyCADG}*{Chap.~IV, Theorem~(13.8)} reads
\[
 E_2^{a,b}=H^a(Z,R^b g_*V\otimes M^p)
 \quad\Longrightarrow\quad H^{a+b}(Y,V\otimes g^*M^p).
\]
No differential leaves \((q,0)\), and the only incoming differentials
have sources \((q-b-1,b)\), \(1\le b<q\). Since \(E_\infty^{q,0}\)
embeds in \(H^q(Y,V\otimes g^*M^p)\), the exact sequence for \(Q\)
gives the following bound, as in Fujita's algebraic argument
\cite{FujitaSemipositive}*{proof of Theorem~(6.2), pp.~370--371}.
\begin{equation}\label{eq:support-leray}
\begin{aligned}
 h^q(Z,F\otimes M^p)
 &\le h^q(Y,V\otimes g^*M^p)+h^{q-1}(Z,Q\otimes M^p)\\
 &\quad+\sum_{b=1}^{q-1}h^{q-b-1}(Z,R^b g_*V\otimes M^p).
\end{aligned}
\end{equation}
Omit terms of negative degree. The first term has the required form
with \(f:Y\to Z\hookrightarrow X\). Each remaining sheaf has smaller
support, and its support dimension minus its displayed degree is at most
\(d-q\), since \((d-1)-(q-1)=(d-b-1)-(q-b-1)=d-q\).
Push these sheaves forward under \(Z\hookrightarrow X\) and apply
induction on \(X\). This gives \eqref{eq:support-comparison} with
all sheaves and maps independent of \(p\).
\end{proof}

\section{Section maps and the endpoint estimates for Theorem~\ref{thm:main}}
\label{sec:section-growth}\label{sec:auxiliary}\label{app:geometric-rank}\label{sec:scalar}
By Siegel's argument \cite{Siegel1955}, in the form given by Ma and
Marinescu, it suffices for the scalar estimate to bound the rank of the
map defined by the sections of \(L^p\). We shall compare its pullback
of the Fubini--Study form with the curvature of \(L\) at a point
where the map has maximal rank.

\begin{proposition}\label{cor:section-growth}\label{geo:semipositive-maximum}
Let \(\mathcal F\ne0\) be a coherent sheaf on \(X\), let
\(d=\dim\operatorname{Supp}\mathcal F\), and put \(r=\crk(L)\).
Then \(h^q(X,\mathcal F\otimes L^p)=O(p^{\min\{r,d\}})\) for \(q=0,d\).
\end{proposition}
\begin{proof}
We first prove the degree-zero bound with \(r(h)\) in place of \(r\),
simultaneously for all compact complex manifolds \(X\), line bundles
\((L,h)\), and coherent sheaves. We induct on the support dimension
\(d\). For \(d=0\) the dimension is the length of the sheaf.

\proofstep{1}{Scalar sections in dimension \(d\)}
Suppose \(d>0\). Take a connected manifold \(X\) of dimension \(d\)
and fix \(p\ge1\) with \(H^0(X,L^p)\ne0\). A basis \(s_0,\ldots,s_N\)
defines the section map \(\Phi_p=[s_0:\cdots:s_N]\) off its common
zero set. Let \(s\) be the maximum rank of this map. Siegel's lemma gives
\(h^0(X,L^p)\le Cp^s\), with \(C\) independent of \(p\)
\cite{MaMarinescu2007}*{Lemma~2.2.6, (2.2.13), pp.~83--84}. We claim that \(s\le r(h)\).

Suppose \(s>r(h)\). Let \(Z\) be the base locus together with the
points where \(\rk d\Phi_p<s\). This is a proper closed analytic
subset. Indeed, if the sections have coefficients \(f_j\) in a
local frame, the matrix with columns \((f_j,\partial f_j)\) has rank
\(1+\rk d\Phi_p\) off the base locus. After reordering the basis,
the ratios \(s_1/s_0,\ldots,s_s/s_0\) have independent differentials
at some point. Their joint image therefore contains an open subset of \(\C^s\).
A polynomial vanishing on that image is zero. It follows, after division
by \(s_0^a\), that the \(\binom{a+s}{s}\) monomials of degree \(a\) in
\(s_0,\ldots,s_s\) are linearly independent in \(H^0(X,L^{pa})\).

Keep \(p\) fixed and let \(a\) grow. By induction on the support
dimension, applied to \(\mathcal O_Z\) on \(X\),
\[
 h^0(Z,L^{pa}|_Z)\le C_p a^{r(h)}<\binom{a+s}{s}
 \qquad(a\ \text{sufficiently large}).
\]
For \(Z=\varnothing\) the left-hand side is zero.
Thus some nonzero section \(\sigma\) in the span of the monomials
vanishes on \(Z\).

Choose a point \(x\) where \(\sum_j|\sigma s_j|^2_{h^{\otimes p(a+1)}}\)
attains its maximum. The maximum is positive, so \(x\notin Z\) and
\(\rk d\Phi_p(x)=s\). In affine coordinates the Fubini--Study form is
\(\omega_{\rm FS}=i\partial\bar\partial\log(1+|z|^2)\).
Since \(\sigma(x)\ne0\), the logarithm of the squared modulus of its
local holomorphic coefficient is pluriharmonic near \(x\). Taking
\(i\partial\bar\partial\) of the logarithm of the sum at \(x\) gives
\[
 0\le(\Phi_p^*\omega_{\rm FS})_x\le p(a+1)(\Theta_{L,h})_x .
\]
Since \(\omega_{\rm FS}\) is positive definite, this inequality
gives \(s=\rk d\Phi_p(x)\le\rk_x\Theta_{L,h}\le r(h)\), a contradiction. Siegel's lemma
therefore gives \(h^0(X,L^p)=O(p^{r(h)})\). On a disconnected
manifold we sum over its finitely many components.

\proofstep{2}{Sections with coherent coefficients}
By Lemma~\ref{lem:support-reduction} with \(q=0\), it suffices to
bound \(h^0(Y,V\otimes M^p)\) on each connected component \(Y\)
of the manifolds in \eqref{eq:support-comparison}, with \(M=f^*L\).
Here \(\dim Y\le d\), and \eqref{eq:pullback-rank} bounds the curvature
rank of the pullback metric by \(b=\min\{r(h),\dim Y\}\).

If \(\dim Y<d\), apply induction on the support dimension.
Suppose \(\dim Y=d\). Among the maps \(M^{-a}\to V\) for all \(a\ge1\), choose a family
of maximal size whose images are independent lines on a common dense
open set (cf.\ \cite{FujitaSemipositive}*{\S5.10, p.~370}). The empty family is allowed, and its
size is at most \(\rk V\). Writing \(a_i\) for its exponents, its sum embeds
\(G=\bigoplus_i M^{-a_i}\) in \(V\).
Let \(T\) be the torsion subsheaf of \(V/G\), supported in
dimension less than \(\dim Y\). By maximality every map \(M^{-p}\to V\)
vanishes generically in \((V/G)/T\), hence everywhere in this torsion-free
quotient. Thus the map on sections of \(V\otimes M^p\) into
\((V/G)\otimes M^p\) has image in \(H^0(Y,T\otimes M^p)\) and kernel
\(H^0(Y,G\otimes M^p)\). Hence
\[
 h^0(Y,V\otimes M^p)\le \sum_i h^0(Y,M^{p-a_i})+h^0(Y,T\otimes M^p).
\]
Step~1 gives \(h^0(Y,M^{p-a_i})=O(p^b)\), since the \(a_i\) are fixed.
Induction on the support dimension gives \(h^0(Y,T\otimes M^p)=O(p^b)\). This proves the
degree-zero estimate. Choose \(h\) with \(r(h)=r\).

\proofstep{3}{Top degree}
For \(q=d\), \eqref{eq:support-comparison} gives \(q_j=\dim Y_j\)
in every summand. Put \(M_j=f_j^*L\). Serre duality on the compact complex manifold \(Y_j\)
\cite{Serre1955}*{\S13, Th\'eor\`eme~4, p.~22} gives
\[
 H^{q_j}(Y_j,V_j\otimes M_j^p)^*
 \simeq H^0\bigl(Y_j,K_{Y_j}\otimes V_j^*\otimes(M_j^*)^p\bigr).
\]
The curvature of \(M_j^*\) has rank at most \(\min\{r,\dim Y_j\}\)
by \eqref{eq:pullback-rank}. The degree-zero estimate bounds each
summand by \(O(p^{\min\{r,d\}})\), since \(\dim Y_j=q_j\le d\).
Their finite sum gives the top-degree bound.
\end{proof}

\section{Holomorphic restriction and the Dolbeault complex}\label{fi:sec:localization}
Let \(E\) be a holomorphic vector bundle, \(h\) a smooth Hermitian
metric on \(L\), and \(F_p=L^p\otimes E\), \(p\ge1\).
To estimate \(h^q(X,F_p)\) in the intermediate degrees, we express
the identity on \(F_p\)-valued forms, up to chain homotopy, as a sum
of operators that factor through holomorphic restrictions. On harmonic
forms, harmonic projection kills both homotopy terms. Each restriction
admits a Taylor approximation whose error decays exponentially with
the degree. Since the other norms grow at most polynomially in \(p\),
a Taylor degree of order \(1+\log p\) suffices.

\subsection{Holomorphic restriction between nested polyhedra}
\label{fi:sc:restriction}
For an open set \(V\subset\C^m\), write \(A^2(V)\) for its space of
square-integrable holomorphic functions with respect to Lebesgue measure.
Put \(\Delta_t^m=\{z\in\C^m:\max_i|z_i|<t\}\) for \(t>0\).
Let \(g_1,\ldots,g_d\), \(d\ge1\), be holomorphic on a neighborhood
of \(\overline{\Delta_1^m}\).
For \(0<a_j<b_j\), \(1\le j\le d\), set
\[
 \Omega^+=\{z\in\Delta_1^m: |g_j(z)|<b_j,\ 1\le j\le d\},\qquad
 \Omega^-=\{z\in\Delta_1^m: |g_j(z)|<a_j,\ 1\le j\le d\}.
\]
For the Taylor expansion in \(z\) and the defining functions, assume \(\Omega^+\subset\Delta_\sigma^m\) with \(0<\sigma<1\) and
\begin{equation}\label{eq:nested-domain-ratios}
 \max_{1\le j\le d}\frac{a_j}{b_j}\le\theta<1 .
\end{equation}
Put \(G=(g_1/b_1,\ldots,g_d/b_d)\).

\begin{lemma}[Finite-rank approximation of restriction]\label{fi:lem:restriction}
Fix \(\vartheta\) with \(\max\{\sigma,\theta\}<\vartheta<1\) and set
\(B=\sup_{\Delta_1^m}|G|\). For each integer \(D\ge0\), there is a
linear map \(R_D:A^2(\Omega^+)\to A^2(\Omega^-)\) approximating the
restriction \(R\), with
\begin{equation}\label{fi:eq:restriction}
 \rk R_D\le\binom{D+m+d}{m+d},\qquad
 \|R-R_D\|\le C(1+B)^d\vartheta^D .
\end{equation}
The constant depends only on \(m,d,\sigma,\theta,\vartheta\).
\end{lemma}
\begin{proof}
\proofstep{1}{Extend from the graph}
The map \(T(z,w)=(z,(w-G(z))/(1+B))\) sends \(\Delta_1^{m+d}\) biholomorphically
onto a pseudoconvex domain \(V\subset\Delta_1^{m+d}\).
It takes the graph of \(G\) to the coordinate slice
\(V\cap\{w'=0\}=\Omega^+\times\{0\}\).
Extend \(f\in A^2(\Omega^+)\) from the entire slice by Ohsawa's
higher-codimension form \cite{Ohsawa1988}*{Corollary~2, p.~266}
of the Ohsawa--Takegoshi theorem \cite{OT1987}.
The affine subspace estimate \cite{DemaillyKollar2001}*{Theorem~2.1, p.~532}
with zero weight and exponent gives extension norm at most \(C\|f\|\),
since the diameter of \(V\) is uniformly bounded. The complex
Jacobian determinant of \(T\) is \((1+B)^{-d}\). Pulling the extension
back by \(T\) gives
\[
 F(z,G(z))=f(z),\qquad \|F\|_{A^2(\Delta_1^{m+d})}\le C(1+B)^d\|f\|.
\]
Point evaluations are continuous in \(A^2(\Delta_1^{m+d})\), so the
functions vanishing on the graph form a closed subspace. Taking \(F\)
orthogonal to this subspace gives a linear extension of least norm.

\proofstep{2}{Restrict the Taylor polynomial}
Put \(u=\max\{\sigma,\theta\}\), and choose \(r\) with
\(u/\vartheta<r<1\). For \(z\in\Omega^-\), the point
\(z'=(z,G(z))\) belongs to \(\Delta_u^{m+d}\). The mean value
inequality gives \(\sup_{\Delta_r^{m+d}}|F|\le C\|F\|\). Applying
the one-variable Cauchy estimate to \(\lambda\mapsto F(\lambda z')\)
on \(|\lambda|=r/u\) and summing the tail gives
\[
 |F(z')-(T_DF)(z')|\le C\|F\|\,(u/r)^{D+1}/(1-u/r),
\]
where \(T_DF\) is the Taylor polynomial of total degree \(D\) at zero.
Set \(R_D f(z)=(T_DF)(z,G(z))\). Its image is spanned by at most
\(\binom{D+m+d}{m+d}\) restricted monomials. Integrating the error
on \(\Omega^-\subset\Delta_1^m\) gives \eqref{fi:eq:restriction}.
\end{proof}

\subsection{Restriction identities on nested domains}
\label{fi:sc:operator-data}\label{fi:sc:intrinsic}
Fix Hermitian metrics on \(X\) and \(E\), a volume form on \(X\),
and the product metric on \(F_p\). On a Hermitian domain \(U\)
with a volume form and a holomorphic map \(\rho:U\to X\), we use
forms with values in \(\rho^*F_p\). For a domain in \(X\),
take \(\rho\) to be the inclusion and use the induced metrics. Put
\[
 \begin{aligned}
 \mathscr L_p^b(U)&=L^2(U,\Lambda^{0,b}T^*U\otimes\rho^*F_p),\\
 \mathscr D_p^b(U)&=\{u\in\mathscr L_p^b(U):
       \db u\in\mathscr L_p^{b+1}(U)\ \text{distributionally}\}.
 \end{aligned}
\]
Give the maximal domain \(\mathscr D_p^b(U)\) its graph norm, and
set both spaces to zero outside the range of form degrees on \(U\).
For \(U^-\subset U^+\), write \(\mathscr L_\pm^b,\mathscr D_\pm^b\)
for the spaces on \(U^\pm\) and \(\db_\pm\) for their differentials.
Restriction \(R\) commutes with \(\db\) on the maximal domains.
With induced metrics \(R\) has norm at most one for both the \(L^2\)
and the graph norm. For different metrics we assume both norms of
\(R\) are bounded by \(Cp^M\). Taking \(U^-=U^+=X\) gives
\(R=\operatorname{id}\).

\paragraph{Local homotopies.}\label{fi:par:operator-assumptions}
For \(0\le b\le n\), an identity for \(R\) has the form
\begin{equation}\label{fi:eq:restriction-homotopy}
 R-K_b=\db_-T_b+T_{b+1}\db_+\quad\text{on }\mathscr D_+^b,\qquad
 K_b=\sum_\iota\outmap_{b\iota}\restr_\iota\inmap_{b\iota}.
\end{equation}
Here \(T_b:\mathscr L_+^b\to\mathscr D_-^{b-1}\) is defined for
\(0\le b\le n+1\), with \(T_0=0\). If \(e=\rk E\), each term in
\(K_b\) factors through a pair of polyhedra
\(\Omega_\iota^-\subset\Omega_\iota^+\) by the maps
\[
 \mathscr L_+^b\xrightarrow{\ \inmap_{b\iota}\ } A^2(\Omega_\iota^+)^{\oplus e}
 \xrightarrow{\ \restr_\iota\ } A^2(\Omega_\iota^-)^{\oplus e}
 \xrightarrow{\ \outmap_{b\iota}\ }\mathscr L_-^b .
\]
Here \(\restr_\iota(f_1,\ldots,f_e)=
(f_1|_{\Omega_\iota^-},\ldots,f_e|_{\Omega_\iota^-})\). We impose the hypotheses of Lemma~\ref{fi:lem:restriction}
uniformly on these pairs. Assume their dimensions and numbers of defining
functions are bounded, the margins \(\sigma,\theta\) are fixed, and
\(B=\sup_{\Delta_1^m}|G|\le Cp^A\). The levels \(a_j,b_j\) lie in
a fixed compact subinterval of \((0,\infty)\). Thus the Taylor
estimate \eqref{fi:eq:restriction} applies to each component.

Assume also that the number of terms, counted with multiplicity,
the norms of \(\inmap_{b\iota}\) and \(\outmap_{b\iota}\), and the
graph norm bound for \(T_b\) grow at most polynomially in \(p\),
uniformly in the form degree. This bounds \(K_b\) on \(L^2\).
For \(0\le b\le n\), the right-hand side of
\(T_{b+1}\db_+=R-K_b-\db_-T_b\) then gives a polynomially bounded
extension of \(T_{b+1}\db_+\) to \(\mathscr L_+^b\).

\paragraph{From holomorphic restrictions to cohomology.}
Cartan and Serre deduce finite-dimensionality from the compactness of
holomorphic restrictions \cite{CartanSerre1953}. A bound for the
ranks of their Taylor approximants gives a bound for the dimension.
\begin{proposition}[Holomorphic localization]
\label{fi:prop:localization}\label{fi:new:cor:observation}
Let \(F\) be a Hermitian holomorphic vector bundle on \(X\) and
\(0\le q\le n\). Suppose its Dolbeault complex has a degree-\(q\) identity
\eqref{fi:eq:restriction-homotopy} for \(R=\operatorname{id}\), whose residual operator factors through finitely many holomorphic
restrictions with bounded input and output maps. The holomorphic
spaces may be weighted. With the
Hilbert direct sum norms, put
\[
 Uu=(\inmap_{q\iota}u)_\iota,\qquad
 S=\bigoplus_\iota\restr_\iota,\qquad
 V\bigl((f_\iota)_\iota\bigr)=\sum_\iota\outmap_{q\iota}f_\iota .
\]
If the restrictions are Hilbert--Schmidt, then
\begin{equation}\label{eq:localization-hs}
 h^q(X,F)\le \|U\|^2\|V\|^2\sum_\iota\|\restr_\iota\|_{\mathrm{HS}}^2 .
\end{equation}
For \(F=L^p\otimes E\), under the hypotheses of
Section~\ref{fi:sc:operator-data} on the unweighted
restrictions, let \(\rescount(p)\) bound their number over all degrees.
Then for constants \(C,M\) independent of \(p,q\),
\begin{equation}\label{fi:eq:localization}
 h^q(X,L^p\otimes E)\le C\rescount(p)(1+\log p)^M,\qquad 0\le q\le n.
\end{equation}
\end{proposition}
\begin{proof}
Let \(P_q\) be orthogonal projection onto the harmonic \(F\)-valued
\((0,q)\)-forms \(\mathcal H^q\), which represent \(H^q(X,F)\)
\cite{MaMarinescu2007}*{Theorems~1.4.1 and~B.4.4}.
Harmonic forms are closed and \(P_q\) kills exact forms. Thus
\(P_qVSUu=u\) for \(u\in\mathcal H^q\). Taking Hilbert--Schmidt norms
on \(\mathcal H^q\) gives \eqref{eq:localization-hs}, since
the squared Hilbert--Schmidt norm of its identity is its dimension. Also, for any
finite-rank approximation \(S'\) to \(S\),
\begin{equation}\label{eq:finite-rank-principle}
 \|V\|\,\|S-S'\|\,\|U\|<1
 \quad\Longrightarrow\quad h^q(X,F)\le\rk S' ,
\end{equation}
since \(P_qVS'U\) is then injective on \(\mathcal H^q\).

For the unweighted restrictions fix \(\max\{\sigma,\theta\}<\vartheta<1\).
By Lemma~\ref{fi:lem:restriction} the direct sum \(S_D\) of the
degree-\(D\) Taylor approximations has error at most \(Cp^A\vartheta^D\)
for some \(A\), and rank at most \(C\rescount(p)(1+D)^M\).
Here the error is the maximum of the individual errors.
The factor norms and their number have polynomial bounds in \(p\),
so \(\|U\|\|V\|\) does as well. Taking
\(D=\lceil C_0(1+\log p)\rceil\) with \(C_0\) sufficiently large
gives \eqref{fi:eq:localization} by \eqref{eq:finite-rank-principle}.
\end{proof}

\subsection{Composition over a cover}
\label{fi:sc:composition}\label{fi:sc:domains}\label{app:cochain-composition}

We now pass from the restrictions on the intersections of a cover to
restriction on the domain being covered. For each intersection we use
just two domains, one for the input and one for the output of its
local homotopy. We choose them according to the \v{C}ech degree so
that the output on one face can be restricted to the input on the
next intersection.

Put \(J=n+1\). For an ordinary cover by domains \(V_i\) with a
subordinate partition of unity \(\{\chi_i\}\), let
\(\mathcal K_a\) consist of the ordered tuples \(I=(i_0,\ldots,i_a)\),
\(0\le a\le J\), whose partition-of-unity supports meet. Repeated
indices are allowed. As each normalized level \(s\) decreases, tighten
the bounds in the defining inequalities of the source \(\Omega\) and the
cover members. On the atlas of \(X\) take \(\Omega(s)=X\). Write
\begin{equation}\label{fi:sc:tupledomain}
 U_I(s)=\Omega(s)\cap V_{i_0}(s)\cap\cdots\cap V_{i_a}(s).
\end{equation}
All earlier defining inequalities are retained. Between the outer
source \(U^+\) and the smaller source \(U^-\) choose decreasing
levels \(\lambda_0>\cdots>\lambda_{J+1}\) as in
Appendix~\ref{app:domain-levels}. For a tuple \(I\in\mathcal K_a\),
take \(U_I(\lambda_a)\) as its larger domain and \(U_I(\lambda_{a+1})\)
as its smaller one.

Let \(C_+^{a,b}\) and \(C_-^{a,b}\) be the Hilbert direct sums, over
these tuples, of the spaces of \(L^2\) forms of degree \(b\) on the
larger and smaller domains, respectively. Use total degree \(a+b\)
and vertical differentials \(d_\pm=(-1)^a\db\) on their maximal
domains. Denote by \(\cR:C_+\to C_-\) restriction within each tuple.
The alternating sum of the face maps gives
\(\delta:C_-^{a,b}\to C_+^{a+1,b}\): both sides are at level
\(\lambda_{a+1}\). Set \(\delta=0\) for \(a=J\). Omitting two
indices in the two possible orders gives the same restriction with
opposite signs. Thus
\[
 \delta\cR\delta=0,\qquad d_+\delta+\delta d_-=0 .
\]
The second identity follows from the sign in the vertical differential.
The total differentials are \(D_+=d_++\delta\cR\) and
\(D_-=d_-+\cR\delta\). We assume polynomial bounds in \(p\) on
the number of tuples and the face norms, so the horizontal
differentials are bounded and the total domains equal the vertical ones.
The same calculation holds for face pullbacks that commute with
\(\db\) and restriction and satisfy the same identities on pairs of faces.

Let \(\iota\) take forms on \(U^+\) to the members of the cover in
\v{C}ech degree zero. Let \(\pi\) map \(C_-\) back to forms on \(U^-\).
Writing \(R\) for restriction from \(U^+\) to \(U^-\), we require
\begin{equation}\label{fi:sc:transfer-data}
 D_+\iota=\iota\db_+,\qquad \db_-\pi=\pi D_-,\qquad \pi\cR\iota=R .
\end{equation}
These identities hold on the maximal domains through total degree \(n\).
We require \(\iota,\pi\) to preserve those domains and to have
polynomial \(L^2\) norms through degree \(n+1\).

\begin{lemma}[Composition over a cover]\label{fi:lem:composition}\label{fi:new:prop:analyticinduction}
Suppose each tuple restriction has an identity
\eqref{fi:eq:restriction-homotopy} through form degree \(n\), with
the bounds of Section~\ref{fi:sc:operator-data} and graph norm
bounds for the homotopies through degree \(n+1\). Then the source
restriction has an identity with the same properties,
\begin{equation}\label{fi:sc:composite-identity}
 R-K_{\rm src}=\db_-T+T\db_+ .
\end{equation}
If \(M_I\) is the number of holomorphic restriction factors for tuple
\(I\), counted with multiplicity over the form degrees, at most
\begin{equation}\label{fi:new:eq:restriction-sum}
 \sum_{a=0}^{n}\ \sum_{I\in\mathcal K_a} M_I
\end{equation}
factors occur on the source.
\end{lemma}
\begin{proof}
Multiply each local homotopy by \((-1)^a\) in \v{C}ech degree \(a\).
Denote their direct sum by \(H:C_+\to C_-\), of degree \(-1\), and
set \(H=0\) in form degrees above \(n+1\). In total degrees at most
\(n\), the direct sum \(K\) of the local residual operators satisfies
\[
 \cR-K=d_-H+Hd_+.
\]
To include the face differential set
\[
 h=H-H\delta H+\cdots+(-1)^J H(\delta H)^J .
\]
The sum stops because \(\delta\) raises the \v{C}ech degree.
Consequently \(H(1-\delta h)=h=(1-h\delta)H\). Expanding and using
\(\delta\cR\delta=0\) and \(d_+\delta+\delta d_-=0\) gives
\begin{equation}\label{fi:sc:additivefactorization}
 (1-h\delta)K(1-\delta h)=\cR-D_-h-hD_+ .
\end{equation}
Indeed, the term with \(\cR\) expands to
\(\cR-h\delta\cR-\cR\delta h\), and the remaining mixed terms
sum to \(h(\delta d_-+d_+\delta)h=0\).
All products are defined on the maximal domains: the face maps
commute with \(\db\), and \(H\) maps \(L^2\) to the vertical graph
domain. The finite sum for \(h\) has polynomial graph norm through
input degree \(n+1\).

Compose \eqref{fi:sc:additivefactorization} with \(\iota\) and \(\pi\).
By \eqref{fi:sc:transfer-data} we may take
\[
 T=\pi h\iota,\qquad K_{\rm src}=\pi(1-h\delta)K(1-\delta h)\iota.
\]
For an input of degree \(n+1\), \(h\iota\) takes values in degree
\(n\), where \(\pi\) still commutes with the differential. This gives
the graph norm bound for \(T\) in that degree. Each local factor of
\(K\) now appears once, preceded and followed by bounded maps. Its
total degree is fixed by its bidegree, so the count is
\eqref{fi:new:eq:restriction-sum}.
\end{proof}

\paragraph{The ordinary cover.}\label{fi:sc:global}
Choose a partition of unity \(\{\chi_i\}\) on the outer source with
\(\supp_\Omega\chi_i\subset V_i(\lambda_*)\) at an inner level
\(\lambda_*<\lambda_{J+1}\). The map back to the smaller source is
the Dolbeault partition-of-unity map
\cite{BottTu1982}*{\S8, (8.6), and \S9, Proposition~9.5},
\begin{equation}\label{fi:sc:partition-map}
 \pi_a u=\sum_{I=(i_0,\ldots,i_a)\in\mathcal K_a}
 \chi_{i_0}\db\chi_{i_1}\wedge\cdots\wedge\db\chi_{i_a}\wedge u_I .
\end{equation}
The two face restrictions of a source form agree, and
\(\pi\cR\iota=R\) follows from \(\sum_i\chi_i=1\).
In \(\pi_{a+1}\cR\delta u\), summing over the first omitted index
gives \(\sum_i\chi_i=1\), hence the term obtained by differentiating
the first cutoff in \(\db\pi_a u\). For every other omission the sum
over that index gives \(\sum_i\db\chi_i=0\). Passing \(\db\) across
\(a\) one-forms gives the sign \((-1)^a\) on \(\db u_I\).
Over \(U^-\) the products of cutoff supports lie inside the tuple
domains in use, so the identity holds distributionally. The tuple counts and cutoff
derivatives give polynomial norm bounds. For the atlas the source
restriction is the identity on \(X\).

\section{Taylor expansion of the local weight}\label{fi:sec:parabolic}

Let \(\phi\) be the weight of \(L\) in a local holomorphic frame.
Multiplying the frame by \(e^g\) replaces \(\phi\) by \(\phi-2\Rea g\).
We seek domains on which \(p(\phi-2\Rea g)\) stays bounded. Their
widths will depend on the Levi matrix
\(\boldsymbol H_\phi=\Levi_z\phi\), for which we use the convention
\begin{equation}\label{half:eq:matrix-convention}
 (\boldsymbol H_\phi)_{jk}=\partial_{\bar z_j}\partial_{z_k}\phi,\qquad
 i\partial\bar\partial\phi
 =i\sum_{j,k}(\boldsymbol H_\phi)_{jk}\,dz_k\wedge d\bar z_j .
\end{equation}
At a point of locally constant Levi rank \(m\), we take width \(t\)
along the kernel at that point and \(t^2\) in the transverse
directions. After a change of frame the first term that can survive
this rescaling has order four. We shall express its coefficient in
the form \(Q^*\boldsymbol\Sigma Q\), where \(Q\) is holomorphic
and \(\boldsymbol\Sigma\) is a signature matrix, that is, a diagonal matrix
with entries \(1\) or \(-1\).
With our convention, \eqref{half:eq:matrix-convention} gives
\[
 \boldsymbol H_{Q^*\boldsymbol\Sigma Q}=(DQ)^*\boldsymbol\Sigma DQ
\]
for every holomorphic map \(Q\) and every constant Hermitian matrix
\(\boldsymbol\Sigma\) of the appropriate size. Here \(^{*}\) denotes
conjugate transpose. We use the Euclidean norm and its induced
operator norm.

\subsection{Bounded kernel factorization}\label{fi:new:sec:scalar}

A Hermitian polarization of a real function \(P\) on a polydisc is a
kernel \(P(z,\bar w)\), holomorphic in \(z\) and antiholomorphic in \(w\),
such that \(P(z,\bar w)=\overline{P(w,\bar z)}\) and \(P(z,\bar z)=P(z)\).
For a base point \(z_*\), set
\begin{equation}\label{fi:it:eq:mixed-projection}
 \mathcal P_{z_*}P(z,\bar w)
 =P(z,\bar w)-P(z,\bar z_*)-P(z_*,\bar w)+P(z_*,\bar z_*).
\end{equation}
On the diagonal, this subtracts
\(2\Rea\bigl(P(z,\bar z_*)-\tfrac12P(z_*,\bar z_*)\bigr)\) from
\(P\). In particular, \(\mathcal P_{z_*}\) annihilates pluriharmonic
functions. Its four terms give an operator norm at most four for the
supremum norm on \(D\times D\) whenever \(z_*\in D\).

The other identity we need follows by block elimination. For a
Hermitian matrix of rank at most \(m\) with invertible
\(m\times m\) block \(A\),
\begin{equation}\label{fi:oa:Schur-jets}
 \operatorname{rank}\begin{pmatrix}A&B\\ B^*&D\end{pmatrix}
 =m+\operatorname{rank}(D-B^*A^{-1}B),\qquad D=B^*A^{-1}B.
\end{equation}
\begin{lemma}[Bounded factorization]\label{fi:new:lem:fekete}
Let \(z=(u,v)\) range over a product polydisc, with \(u\in\C^m\),
\(m\ge1\), and fix \(z_*=(0,v_*)\). Suppose \(P\) has a Hermitian
polarization on this product. Assume its Levi matrix has rank at most
\(m\) and, up to a pluriharmonic term, \(P\) has the form
\[
 P(u,v)=u^*Au+2\Rea\bigl(u^*h(v,\bar v)\bigr)+k(v,\bar v)
\]
with \(A\) a constant invertible Hermitian matrix. Write
\(\mathcal M=\mathcal P_{z_*}P\). For every polydisc \(D\) containing
\(z_*\) and relatively compact in the given product, there exist a
signature matrix \(\boldsymbol\Sigma\) and a holomorphic map
\(Z(u,v)=Cu+H(v)\), with \(C\) invertible, such that
\begin{equation}\label{fi:new:eq:fekete}
 \mathcal M(z,\bar w)=Z(w)^*\boldsymbol\Sigma Z(z),\qquad
 Z(z_*)=0,\qquad
 \sup_D|Z|\le m\sqrt{\sup_{D\times D}|\mathcal M|}.
\end{equation}
Each component of \(Z\) is a linear combination of \(m\) functions
\(\mathcal M(\,\cdot\,,\bar w_j)\) with \(w_j\in\overline D\).
\end{lemma}
\begin{proof}
\proofstep{1}{Complete the square}
The upper row of blocks of \(\Levi P\) is
\((A,\partial_v h)\), where
\(\partial_v h=(\partial h_i/\partial v_j)_{i,j}\). Thus
\[
 \Levi_v P=(\partial_v h)^*A^{-1}\partial_v h
\]
by \eqref{fi:oa:Schur-jets}. Comparing the coefficients of \(\bar u_i\)
gives \(\Levi_v h_i=0\), since the right-hand side is independent
of \(u\). Thus \(h=a+\bar b\) on the polydisc, with
\(a,b\) holomorphic and \(a(v_*)=0\). The term
\(2\Rea(u^*\bar b)\) is pluriharmonic. Setting \(u=0\) gives
\(\Levi_v(k-a^*A^{-1}a)=0\). Completing the square gives
\(P=\Phi^*A\Phi\) up to a pluriharmonic term, where
\(\Phi(u,v)=u+A^{-1}a(v)\). Since \(\Phi(z_*)=0\), its
polarization gives
\begin{equation}\label{fi:new:eq:qualitativekernel}
 \mathcal M(z,\bar w)=\Phi(w)^*A\Phi(z).
\end{equation}

\proofstep{2}{Bound the holomorphic factors}
Let \(S\) have columns \(\Phi(w_1),\ldots,\Phi(w_m)\).
Choose the points \(w_i\in\overline D\) to maximize \(|\det S|\).
This determinant is nonzero since \(\Phi(u,v_*)=u\). Put
\(c(z)=S^{-1}\Phi(z)\). By Cramer's rule, \(c_i(z)\) is the ratio
of the determinant with column \(i\) replaced by \(\Phi(z)\) to
\(\det S\). Maximality gives \(|c_i(z)|\le1\) for \(z\in D\), as in
\cite{BosEtAl2011}*{\S3, (3.1)}.

Now \(B=S^*AS\) has entries \(\mathcal M(w_j,\bar w_i)\), hence
\(\|B\|\le m\sup_{D\times D}|\mathcal M|\). Write
\(B=U^*\Lambda U\) with \(U\) unitary and set
\(Z=|\Lambda|^{1/2}Uc\), \(\boldsymbol\Sigma=\operatorname{sgn}\Lambda\).
The kernel identity follows from \(B=S^*AS\), and
\(|Z|\le\sqrt{\|B\|}\,|c|\) with \(|c|\le\sqrt m\) gives the bound
in \eqref{fi:new:eq:fekete}. Since \(Z\) is an invertible linear
transform of \(\Phi\), it has the asserted form. Finally \(Bc(z)\)
has entries \(\mathcal M(z,\bar w_i)\). Inverting \(B\) expresses
the components of \(Z\) as linear combinations of these functions.
\end{proof}

\subsection{The fourth-order term after rescaling}

Let \(T_s\phi\) denote the Taylor polynomial of total degree \(s\)
at the center. We use \(C^j\) norms in the real coordinates, with
derivatives through order \(j\).

\begin{lemma}[Parabolic rescaling of the local weight]\label{fi:oa:parabolic}\label{fi:it:lem:initial}
Let \(\phi\) be a real \(C^5\) function on a fixed neighborhood of
\(0\in\C^n\). Suppose its Levi rank is constant and equal to \(m>0\)
near \(0\). Choose unitary coordinates \(z=(x,y)\) in which the
\(y\)-directions span the Levi kernel at \(0\). Let \(\alpha\)
range over a fixed bounded subset of \(\C^m\), and \((u,v)\) over a fixed
product box in \(\C^m\times\C^{n-m}\).
There exists \(t_0>0\) such that for each \(\alpha\) there are a
signature matrix \(\boldsymbol\Sigma_\alpha\) and a polynomial map
\(Q_\alpha:\C^n\to\C^m\) with the following property. For every
\(0<t\le t_0\), there is a holomorphic polynomial \(g_{t,\alpha}\) such that
\begin{equation}\label{fi:oa:parabolic-error}
 \phi(t\alpha+t^2u,tv)
 =2\operatorname{Re}g_{t,\alpha}(u,v)
   +t^4Q_\alpha(u,v)^*\boldsymbol{\Sigma}_\alpha Q_\alpha(u,v)+O(t^5).
\end{equation}
The map has the form \(Q_\alpha(u,v)=\boldsymbol C_\alpha u+G_\alpha(v)\)
with \(\boldsymbol C_\alpha\) invertible, \(\deg Q_\alpha\le2\), and
\(Q_\alpha(0)=0\). The choice of \(t_0\) and the coefficient and
remainder bounds depend only on the fixed sets and the \(C^5\) bound
for \(\phi\).

If the Levi form vanishes near the center, a holomorphic change of
frame removes \(T_4\phi\), leaving a remainder \(O(t^5)\) on a box
of width \(t\).
\end{lemma}

\begin{proof}
\proofstep{1}{Locate the first mixed Taylor coefficient}
Set \(S_t(u,v)=(t\alpha+t^2u,tv)\) and write \(\Levi\phi\) in
blocks \(A,B,B^*,D\), with \(B\) in the upper right. Our choice
of kernel coordinates gives \(A(0)=A_0\) invertible and \(B(0)=0\).
The upper blocks of \(\Levi(\phi\circ S_t)\) are consequently
\(t^4A(S_t)=t^4(A_0+O(t))\) and \(t^3B(S_t)=O(t^4)\).
For the fixed center and sufficiently small \(t\), the rank is at
most \(m\) and the inverse of the first block is \(O(t^{-4})\).
Substitution in \eqref{fi:oa:Schur-jets} gives \(O(t^4)\) for
the lower right block as well.

Write \((T_4\phi)\circ S_t=\sum_{j=0}^8t^jP_j\). Since \(S_t\)
is affine and \(|S_t|+\|DS_t\|=O(t)\), Taylor's formula gives
\(\|(\phi-T_4\phi)\circ S_t\|_{C^2}=O(t^5)\). Comparing the
coefficients of the Levi matrices and using the closedness of the
rank bound yields
\begin{equation}\label{fi:oa:Levi-limit}
 \Levi P_j=0\ (j<4),\qquad
 \Levi P_4=\lim_{t\to0}t^{-4}\Levi(\phi\circ S_t),\qquad
 \rk\Levi P_4\le m .
\end{equation}

\proofstep{2}{Factor the fourth coefficient}
A factor \(u\) or \(\bar u\) carries weight \(t^2\). Thus the terms
of \(P_4\) quadratic in these variables have constant coefficients,
and its mixed quadratic term is \(u^*A_0u\). After discarding pure
terms, what remains besides \(u^*A_0u\) is
\(2\Rea(u^*h(v,\bar v))+k(v,\bar v)\), with \(\deg h\le2\).
Lemma~\ref{fi:new:lem:fekete}, applied at \(0\) on a fixed polydisc
containing the required enlargement, gives
\begin{equation}\label{fi:oa:qualitative-square}
 (\mathcal P_0P_4)(z,\bar w)
   =Q_\alpha(w)^*\boldsymbol\Sigma_\alpha Q_\alpha(z),\qquad
 Q_\alpha(u,v)=\boldsymbol C_\alpha u+G_\alpha(v).
\end{equation}
The factorization gives \(\boldsymbol C_\alpha\) invertible and
\(Q_\alpha(0)=0\). In Step~1 of Lemma~\ref{fi:new:lem:fekete} the
holomorphic part of \(h\) has degree at most two, hence so does
\(Q_\alpha\). The \(C^5\) bound on \(\phi\) and the bound on
\(\alpha\) bound the coefficients of \(P_4\). The supremum estimate
\eqref{fi:new:eq:fekete} followed by Cauchy's inequalities therefore
bounds the coefficients of \(Q_\alpha\) uniformly.

For \(j<4\), \eqref{fi:oa:Levi-limit} gives \(P_j=2\Rea h_j\).
Write \(P_4-\mathcal P_0P_4=2\Rea h_4\) on the diagonal. Taking
\(h_j(0)\) real determines these holomorphic polynomials and bounds
their coefficients uniformly. The correction
\(g_{t,\alpha}=\sum_{j=0}^4t^jh_j\) removes all these terms.
The Taylor remainder and the terms of degree at least five in \(t\)
are \(O(t^5)\), which proves \eqref{fi:oa:parabolic-error}.

When the Levi rank is locally zero, \(\phi\) is pluriharmonic
near the center, hence \(T_4\phi\) itself is pluriharmonic. In both
cases Taylor's formula bounds the error on the fixed coordinate
neighborhood, independently of how small the constant-rank
neighborhood is.
\end{proof}

\subsection{The cohomology estimate before iteration}

\begin{proposition}\label{prop:one-scale}
Suppose \(1\le r<n\) and \(L\) has a smooth metric with curvature
rank at most \(r\). For each fixed holomorphic vector bundle \(E\)
there are constants \(C,N\) such that
\begin{equation}\label{eq:one-scale}
 h^q(X,E\otimes L^p)
 \le C p^{\,r+2(n-r)/5}(1+\log p)^N
 \qquad (p\ge1,\ 0\le q\le n).
\end{equation}
\end{proposition}
\begin{proof}
\proofstep{1}{Subdivide the transverse directions}
Take \(t=p^{-1/5}\) and a finite family of coordinate charts with
fixed enlargements and local holomorphic frames for \(L\) and \(E\).
A mesh of width \(t\) gives \(O(t^{-2n})\) boxes. Their enlargements
have bounded overlap, also between charts, because the transition maps
and their inverses have bounded first derivatives. Thus there are
\(O(t^{-2n})\) intersections of cutoff supports in each fixed degree.

Enclose one such intersection in a coordinate ball of radius \(Ct\)
and choose a point of maximal Levi rank \(m\) on the ball. If \(m=0\),
the weight is pluriharmonic near that point. Subtracting its Taylor
polynomial of degree four by a change of frame leaves weight
\(O(pt^5)=O(1)\) for \(L^p\) on the enlarged box.

If \(m>0\), an \(m\times m\) minor stays nonzero near the chosen
point, so the rank there is constant. Take unitary coordinates
\((x,y)\) centered at the point, with \(y\) along the kernel, and
subdivide the \(x\)-box at width \(t^2\). This gives \(O(t^{-2m})\)
intersections in each fixed degree. On each enlarged box we now have
\(x=t\alpha+t^2u,\ y=tv\), with \(\alpha,u,v\) in the bounded sets
of Lemma~\ref{fi:oa:parabolic}.

\proofstep{2}{Choose bounded frames on the remaining domains}
For \(m>0\) put \(P=pt^4\) and abbreviate \(Q_\alpha\),
\(\boldsymbol\Sigma_\alpha\) to \(Q,\boldsymbol\Sigma\). Since
\(pt^5=1\), subtracting \(2p\Rea g_{t,\alpha}\) in
\eqref{fi:oa:parabolic-error} leaves the weight
\(P Q^*\boldsymbol\Sigma Q+O(1)\). Cover the bounded image of \(Q\)
by polydiscs of width \(\delta=P^{-1/2}\). Their cutoffs pulled
back by \(Q\) have the same overlap bound, and the number of
intersections of bounded length is \(O(\delta^{-2m})=O((pt^4)^m)\).
For a target center \(q_*\) we have
\[
 Q^*\boldsymbol\Sigma Q
  =(Q-q_*)^*\boldsymbol\Sigma(Q-q_*)
    +2\Rea(q_*^*\boldsymbol\Sigma Q)-q_*^*\boldsymbol\Sigma q_* .
\]
On the preimage \(\Omega\) of the enlarged target polydisc,
\(P|Q-q_*|^2=O(1)\). To remove the other terms, let \(e\) be
the original local frame, \(|e|_h^2=e^{-\phi}\), and take the
frame \(e^{pg}e^{\otimes p}\) of \(L^p\), where
\[
 pg=p g_{t,\alpha}
      +P(q_*^*\boldsymbol\Sigma Q-\tfrac12q_*^*\boldsymbol\Sigma q_*).
\]
Its weight is \(P(Q-q_*)^*\boldsymbol\Sigma(Q-q_*)+O(1)\). Thus
\begin{equation}\label{fi:eq:gaugebound}
 \sup_{\Omega}|p(\phi-2\Rea g)|\le C .
\end{equation}
The same subdivision applies to any coefficient \(P\ge1\) with
a uniformly bounded remainder.

\proofstep{3}{Count the restrictions}
Compose the coordinate, transverse and target covers, keeping the
defining inequalities at each stage. Lemma~\ref{fi:lem:composition}
requires a fixed number of nested radii for these three covers.
Use fixed multiples of the mesh widths as in
\eqref{eq:nested-domain-ratios}. The cutoff supports lie in the
smallest boxes, with the target cutoffs pulled back by \(Q\).

On the enlarged coordinate polydiscs the coordinate changes have
bounded derivatives and \(Q\) has bounded coefficients. Normalizing
the boxes introduces powers of \(t^{-1}=p^{1/5}\) and
\(\delta^{-1}=p^{1/10}\). The normalized defining functions, the
derivatives of the cutoffs, and the factors in the form and volume
norms therefore have polynomial bounds in \(p\).

Use the frame of \(L^p\) constructed in Step~1 or Step~2, according
as \(m=0\) or \(m>0\), together with the fixed local frame of \(E\).
For a bundle-valued form \(u\), let \(\|u\|_0\) be the Euclidean
\(L^2\) norm of its coefficients and \(\|u\|_{\rm geom}\) its
bundle \(L^2\) norm. The coordinate bounds and \eqref{fi:eq:gaugebound}
give, uniformly in the form degree,
\begin{equation}\label{fi:sc:normequivalence}
 (Cp^M)^{-1}\|u\|_0\le\|u\|_{\rm geom}\le Cp^M\|u\|_0 .
\end{equation}
Since the frames are holomorphic, we may solve \(\db\) coefficient
by coefficient on each polyhedron. By \eqref{fi:sc:normequivalence}
the scalar \(L^2\) estimate gives a polynomial bound in the bundle norm
for the local homotopy on \(U^-\subset U^+\)
(Appendix~\ref{fi:sc:contraction}, \eqref{fi:sc:local-homotopy}).
Each outer intersection gives at most
\(Ct^{-2m}(pt^4)^m=C(pt^2)^m\) restrictions by
Lemma~\ref{fi:lem:composition}, including the case \(m=0\).
There are \(O(t^{-2n})\) outer intersections. Since \(m\le r\)
and \(pt^2\ge1\), Proposition~\ref{fi:prop:localization} gives
\[
 h^q(X,E\otimes L^p)\le C\,t^{-2n}(pt^2)^r(1+\log p)^N .
\]
Now \(t=p^{-1/5}\) gives \eqref{eq:one-scale}.
\end{proof}

In the expression \(t^{-2n}(pt^2)^r=p^rt^{-2(n-r)}\), the
last factor counts the directions still subdivided at width \(t\).

\section{Product maps from complex cells}\label{fi:cell:sec:cells}

In Proposition~\ref{prop:one-scale} it sufficed to subdivide the
image of \(Q(u,v)=\boldsymbol C u+G(v)\). To rescale the remainder in
\eqref{fi:oa:parabolic-error}, we also need to vary the source point
holomorphically while \(Q\) is fixed. Recovering \(u\) from
\((Q,v)\) would require \(\boldsymbol C^{-1}\), whose norm can be arbitrarily large.
Instead we parametrize the graph in the order \((Q,z)\). This gives
products on which \(Q\) depends only on the first group of parameters.

For a scalar holomorphic function, its algebraic complexity is the
least degree of an algebraic variety containing its graph as an
analytic component over its domain. For a map we take the maximum
over its coordinate functions \cite{BN19}*{\S2.2.3}. Fix integers
\(1\le m\le N\), polydiscs
\(\Delta^-\Subset\Delta\Subset\Delta^+\subset\C^N\), and bounds
\(M,D>0\). As in
Section~\ref{fi:sc:restriction}, \(\Delta_a^d\) denotes a polydisc
of radius \(a\), with \(\Delta_a^0\) a point.

\begin{proposition}[Product maps]\label{fi:cell:prop:targetproducts}
Suppose \(Q:\Delta^+\to\C^m\) is a holomorphic algebraic map
with \(\sup_{\Delta^+}|Q|\le M\) and graph complexity at most \(D\).
For each integer \(p\ge1\) there are at most \(C_0(1+\log p)^\ell\)
algebraic holomorphic maps
\[
 f_j:\Delta_2^{d_j}\times\Delta_2^{e_j}\longrightarrow\Delta,\qquad
 Q(f_j(b,V))=q_j(b),\qquad d_j\le m,\quad e_j\le N.
\]
Every point of \(\Delta^-\) is within \((40p)^{-1}\) of some
\(f_j(\Delta_1^{d_j}\times\Delta_1^{e_j})\). The graph degrees of
\(f_j,q_j\) and the constants \(C_0,\ell\) are uniform in \(p,Q\).
\end{proposition}
Put \(\epsilon=(40p)^{-1}\). Only the coverage of \(\Delta^-\)
has error \(\epsilon\). The identity \(Q(f_j(b,V))=q_j(b)\) holds
on the whole product of radius two. There \(f_j\) takes values in
\(\Delta\) and \(|q_j|\le M\), so Cauchy's inequalities bound their
derivatives of each fixed order on every intermediate product.

\begin{proof}[Proof of Proposition~\ref{fi:cell:prop:targetproducts}]
\proofstep{1}{Cover the graph with the target coordinates first}
Choose a source polydisc containing \(\Delta^-\) with fixed
enlargement in \(\Delta\), and a target polydisc containing
\(Q(\Delta^-)\) at a fixed positive distance from its boundary.
In the normalized coordinates on the enlarged product, the graph
is the common zero set of
\[
 (q,z)\longmapsto q_i-Q_i(z),\qquad 1\le i\le m .
\]
Apply the cellular parametrization theorem \cite{BN19}*{Theorem~8}
to these functions, with fixed target extension and source extension
parameter \(\sigma\) as in Appendix~\ref{app:wide-cells}, where
\(1/\sigma=O(1+\log(1/\epsilon))\). If a chart meets the graph,
each pullback has a zero and hence, by compatibility, vanishes
identically. These charts cover the graph over \(\Delta^-\) before
extension. By analytic continuation their extensions still lie in
the graph.

Let \(F\) denote the \(z\)-projection of such a chart. Because the
target coordinates come first, triangularity makes \(Q\circ F\)
a function of the first \(m\) cell coordinates. The theorem gives
at most \(C_0(1+\log p)^\ell\) charts of uniformly bounded degree.
Appendix~\ref{app:cells} gives the same degree bound for their source
radius functions.

\proofstep{2}{Pass from cells to products}
Replace each annular or punctured-disc coordinate by the
holomorphic section in Appendix~\ref{app:wide-cells}. Its Laurent
estimates bound the change in \(F\) at each substitution by
\(\epsilon/(N+m+1)\). There are at most \(N+m\) coordinates,
so the accumulated error is less than \(\epsilon\). The remaining
disc coordinates give a holomorphic map \(f_j\) into the enlarged
source polydisc of Step~1. The substitutions are defined on the
product of radius two, and the product of radius one suffices for
the approximation.

Write \(b\) for the retained coordinates among the first \(m\)
fibers and \(V\) for the others. Thus \(d_j=\dim b\le m\) and
\(e_j=\dim V\le N\). Each substitution in position \(i\) uses
only parameters retained through that position, so on the whole
enlarged product \(Q(f_j(b,V))=Q(f_j(b,0))=:q_j(b)\).
Lemma~\ref{fi:it:lem:complexity} bounds the degrees of the
resulting maps.
\end{proof}

\section{Finite jets and rescaling}\label{fi:it:sec:expansions}

We now repeat the rescaling of Section~\ref{fi:sec:parabolic}.
Fix \(k\ge1\), put \(t=p^{-1/(2k+5)}\), and make all choices of centers
and charts at this mesh. They remain fixed when we take Taylor coefficients
in a second variable \(s\). All uniform bounds below may depend on \(k\).
We work in a coordinate chart centered at a point of locally constant
Levi rank \(m\), using the coordinates \((x,y)\) of
Lemma~\ref{fi:oa:parabolic}.

Let \(P=T_{2k+4}\phi\) be the Taylor polynomial of the original local
weight at the center. On the fixed enlarged charts Taylor's formula gives
\begin{equation}\label{fi:it:eq:physical-remainder}
 |D^j(\phi-P)(z)|\le C_k |z|^{2k+5-j},\qquad j=0,1,2.
\end{equation}
If \(m=0\), then \(P\) is pluriharmonic. Removing it leaves a weight
\(O(pt^{2k+5})=O(1)\) on the box of width \(t\). We suppose henceforth
that \(m>0\) and write \(A_0\) for the invertible \(x\)-block of
\(\Levi\phi(0)\).

At each step we retain a map into this original coordinate chart.
The first such map is
\begin{equation}\label{fi:it:eq:initial}
 S_0(s,u,v)=(s\alpha+s^2u,sv).
\end{equation}
For any map \(S(s,z)\) used below, polarize the \emph{same} polynomial \(P\)
and remove its terms depending on just one of the two variables:
\[
 B_S(s;z,\bar w)=P\bigl(S(s,z),\overline{S(\bar s,w)}\bigr),
 \qquad H_S(s;z,\bar w)=\mathcal P_0 B_S(s;z,\bar w).
\]
Here \(\mathcal P_0\) is given by \eqref{fi:it:eq:mixed-projection}.
In particular \(H_S(s;z,0)=H_S(s;0,\bar w)=0\).
These kernels are holomorphic in \(s\) and in the independent variables
\(z,\bar w\). For real \(s\) the kernels are Hermitian, so their
Taylor coefficients at \(s=0\) are Hermitian as well. By
Lemma~\ref{fi:oa:parabolic},
\[
 H_{S_0}(s;z,\bar w)
 =s^4 Q_0(w)^*\boldsymbol\Sigma_0 Q_0(z)+O(s^5),
\]
where \(Q_0=Q_\alpha\) and \(\boldsymbol\Sigma_0=\boldsymbol\Sigma_\alpha\).

\subsection{One rescaling}\label{fi:it:sec:renormalize}
Suppose that \(\ell<k\) rescalings have been made. Write \(z=(u,v)\)
with \(u\in\C^m\), and put \(d_\ell=4+2\ell\). We use two properties
of the current map \(S_\ell\). First, its dependence on \(u\) has the
explicit form
\begin{equation}\label{fi:it:eq:physical-affine}
 S_\ell(s,u,v)=S_\ell(s,0,v)+s^{\ell+2}(u,0),
 \qquad |S_\ell(s,\cdot)|\le C_k|s| ,
\end{equation}
where \((u,0)\) is in the original \((x,y)\)-coordinates. Second,
its mixed kernel satisfies
\begin{equation}\label{fi:it:eq:expansion}
 H_{S_\ell}(s;z,\bar w)
 =s^{d_\ell}\bigl(Q_\ell(w)^*\boldsymbol\Sigma_\ell Q_\ell(z)+O(s)\bigr).
\end{equation}
The holomorphic map \(Q_\ell\) has \(m\) components and uniformly
bounded supremum norm and algebraic degree. The bounds
hold on fixed enlarged products and for \(|s|\) less than a fixed positive
radius. For \(\ell=0\) both properties follow from the preceding calculation.

Apply Proposition~\ref{fi:cell:prop:targetproducts} on a smaller product
to \(Q_\ell\). On each resulting product chart \(f(b,V)\) we have
\(Q_\ell(f(b,V))=q(b)\). Choose \(b_*\) in an intermediate
polydisc \(D_b^-\) compactly contained in the \(b\)-domain. Take
\(U\in\C^{\dim b}\) and \(W\in\C^m\) in fixed polydiscs about zero
and set
\begin{equation}\label{fi:it:eq:Psi}
 \Psi_{\ell,s}(U,V,W)=f(b_*+sU,V)+s(W,0),\qquad
 S_{\ell+1}(s,\xi)=S_\ell(s,\Psi_{\ell,s}(\xi)),
 \quad \xi=(U,V,W).
\end{equation}
Choose the radius in \(s\) sufficiently small that the extended parameter
product maps a fixed distance inside a common product on which
\(S_\ell,Q_\ell\) and the expansion \eqref{fi:it:eq:expansion} have
their uniform bounds.
This is possible since \(f\) takes its values in a smaller product.
The \(V\)-domain contains \(0\), and \(b_*+sU\) stays in the domain
of holomorphic extension of \(f\). At fixed \(k\) the margins and
the radii in \(s\) can be chosen independently of \(p\) and the chart.

\begin{lemma}[Rescaling]\label{fi:it:lem:renormalize}
The map \(S_{\ell+1}\) satisfies \eqref{fi:it:eq:physical-affine}
and \eqref{fi:it:eq:expansion} with \(\ell+1\) in place of \(\ell\)
and \(W\) as its transverse variable. In the next application use
\((u,v)=(W,(U,V))\). The number of spatial variables increases
by at most \(2m\).
\end{lemma}
\begin{proof}
\proofstep{1}{Identify the first two possible mixed coefficients}
The translation \(s(W,0)\) in \eqref{fi:it:eq:Psi} affects only the
\(u\)-component of the argument of \(S_\ell\). Thus the
dependence on \(W\) in the original coordinates is exactly
\(s^{\ell+3}(W,0)\). This proves \eqref{fi:it:eq:physical-affine}
at the next step.

Write \(\xi'=(U',V',W')\) for the second set of variables. Then
\[
 H_{S_{\ell+1}}(s;\xi,\bar\xi')
 =\mathcal P_0\!\left[
 H_{S_\ell}\bigl(s;\Psi_{\ell,s}(\xi),
              \overline{\Psi_{\ell,\bar s}(\xi')}\bigr)\right].
\]
Here \(\mathcal P_0\) acts in the new variables. A term depending
only on \(z\) or only on \(\bar w\) still vanishes under this projection.
At \(s=0\), \(Q_\ell(\Psi_{\ell,0}(\xi))=q(b_*)\) for every \(\xi\).
After the mixed projection the leading product in \eqref{fi:it:eq:expansion}
therefore vanishes to order two in \(s\): each factor is replaced by
its difference from its value at \(\xi=0\). For the coefficient at
order \(d_\ell+1\), substitute \(f(b_*,V)\) and
\(\overline{f(b_*,V')}\) in the two slots of the coefficient of
\(s^{d_\ell+1}\) in \(H_{S_\ell}\), and apply the new mixed projection.
Consequently
\[
 s^{-d_\ell}H_{S_{\ell+1}}=s A_1+s^2 A_2+O(s^3),
\]
where \(A_1\) depends only on \(V,\bar V'\). The two kernels
\(A_1,A_2\) are centered at zero.

Each occurrence of \(W\) or \(\bar W\) in the original polynomial
substitution carries the factor \(s^{\ell+3}\).
At order \(d_\ell+2=2(\ell+3)\), terms quadratic in \(W,\bar W\)
therefore have constant coefficients. Their mixed part comes from
the original Levi matrix at the center. Up to a pluriharmonic term,
the diagonal of \(A_2\) has the form
\begin{equation}\label{fi:it:eq:model}
 W^* A_0 W+2\Rea\bigl(W^* h(U,V,\bar U,\bar V)\bigr)
                 +a(U,V,\bar U,\bar V).
\end{equation}
Thus the \(W\)-block is the same invertible matrix \(A_0\) as at
the original center, even when \(\rk Dq(b_*)<m\).

\proofstep{2}{Use the rank of the original curvature}
For real \(s>0\), the diagonal of \(H_{S_{\ell+1}}\) differs from
\(P\circ S_{\ell+1}\) by a pluriharmonic term.
By \eqref{fi:it:eq:physical-affine} and Cauchy's inequalities, the
first and second spatial derivatives of \(S_{\ell+1}\) are \(O_k(s)\)
on a smaller product. Thus \eqref{fi:it:eq:physical-remainder} and
the chain rule give
\(\|(\phi-P)\circ S_{\ell+1}\|_{C^2}=O_k(s^{2k+5})\).
Since \(d_\ell+2\le2k+4\), the Levi matrix of
\(s^{-d_\ell}\phi\circ S_{\ell+1}\) has expansion
\(s\Levi A_1+s^2\Levi A_2+O(s^3)\).
For the fixed chart its rank is at most \(m\) for sufficiently small
\(s\), because \(S_{\ell+1}\) then maps into the constant-rank
neighborhood of the original center.

With the \(W\)-variables first, the upper blocks of this matrix are
\(s^2(A_0+O(s))\) and \(O(s^2)\). By \eqref{fi:oa:Schur-jets}
the lower right block is also \(O(s^2)\) for the fixed chart.
It follows that \(\Levi A_1=0\). Since \(A_1\) is centered, its
pluriharmonic diagonal has zero polarization after applying \(\mathcal P_0\);
hence \(A_1=0\). Dividing the Levi matrix by \(s^2\) and taking
\(s\to0\) gives \(\rk\Levi A_2\le m\).

Apply Lemma~\ref{fi:new:lem:fekete} to \eqref{fi:it:eq:model}.
It gives \(A_2(\xi,\bar\xi')=Q_{\ell+1}(\xi')^*
\boldsymbol\Sigma_{\ell+1} Q_{\ell+1}(\xi)\), with \(m\) components
in \(Q_{\ell+1}\).
On a fixed enlargement \(D\), \eqref{fi:new:eq:fekete} gives
\(\sup_D|Q_{\ell+1}|\le m\sqrt{\sup_{D\times D}|A_2|}\).
This supremum and the Taylor remainder are uniform
by Cauchy's inequalities in \(s\). The inverse of \(A_0\) was used
only to identify coefficients at the fixed center. The maps
\(S_{\ell+1}\) are obtained by composing algebraic maps in \(s,\xi\).
For fixed \(k\), extracting the coefficient of order \(d_\ell+2\)
takes at most \(2k+4\) derivatives followed by evaluation at \(s=0\).
Lemma~\ref{fi:it:lem:complexity} bounds its graph degree and hence
that of each component of \(Q_{\ell+1}\), a constant linear
combination of evaluations of this kernel.
Finally \(\dim U\le m\), \(\dim W=m\), and \(\dim V\) is at most the
number of coordinates of \(S_\ell\).
\end{proof}

Iterating the lemma gives \eqref{fi:it:eq:expansion} for \(\ell=k\).
At \(s=t\) put
\(g(z)=B_{S_k}(t;z,0)-\tfrac12 B_{S_k}(t;0,0)\).
By the definition of \(\mathcal P_0\) and \eqref{fi:it:eq:physical-remainder},
the weight in the original chart satisfies
\begin{equation}\label{fi:it:eq:final-expansion}
 \phi(S_k(t,z))
 =2\Rea g(z)+t^{2k+4} Q_k(z)^*\boldsymbol\Sigma_k Q_k(z)
              +O_k(t^{2k+5})
\end{equation}
on the enlarged product. The single holomorphic function \(g\)
accounts for the pluriharmonic part at all preceding steps.

\section{Restriction identities with parameters}
\label{fi:it:sec:relative}
The image of a product map may only approximate a point of the
source. Its graph tube, however, projects onto a neighborhood of
that point. We use these projections to pull back forms and integrate
over the parameter fibers to recover restriction on the source.

Let \(U^-\subset U^+\) be nested levels of a holomorphic polyhedron
\(U\), with coordinates \(x=(z,y)\) and a holomorphic map
\(\rho:U\to X\) into a coordinate chart. Only \(z\) occurs in the Taylor expansion.
For the form and volume norms on \(U^\pm\) we use product coordinates
up to polynomial factors in \(p\); the coefficient metric is always
pulled back from \(X\). We use \eqref{eq:nested-domain-ratios} and polynomial
bounds for the defining functions on fixed enlargements.
At fixed depth \(k\) the numbers of variables and inequalities are
bounded by Appendix~\ref{app:domain-levels}.

Let \(\mathcal A\) be a product containing the \(z\)-projection of
\(U^+\), and let \(f_i:B_i^+\to\C^{\dim z}\), \(1\le i\le K_p\),
be holomorphic maps on parameter products of uniformly bounded
dimension and radius. Assume \(f_i\) is uniformly bounded on a fixed
enlargement of \(B_i^+\), and choose \(B_i^-\Subset B_i^+\) with a
uniform margin. For tubes of width \(\eta=p^{-1}\) assume that there
are constants \(C,L\), independent of \(p\), such that

\begin{equation}\label{fi:it:eq:relativecover}
 \begin{gathered}
 K_p\le C(1+\log p)^L,\\
 \forall z\in\mathcal A\quad
 \exists i,\ \xi\in B_i^-:
       |z_\gamma-(f_i)_\gamma(\xi)|<\eta/20
       \quad(1\le\gamma\le\dim z).
 \end{gathered}
\end{equation}
At level \(\lambda\), let \(U(\lambda)\) and \(B_i(\lambda)\) be the
source and parameter domains and \(c(\lambda)\eta\) the tube width.
Each \(B_i(\lambda)\) contains a fixed neighborhood of \(B_i^-\),
and \(c(\lambda)-1/20\) is bounded below by a fixed positive
constant. We impose \eqref{eq:nested-domain-ratios} at adjacent
levels for every defining inequality. To an ordered tuple
\(I=(i_0,\ldots,i_a)\) associate one independent parameter
\(\xi_j\in B_{i_j}(\lambda)\) for each occurrence of an index. The
intersection over a common source point is
\begin{equation}\label{fi:it:eq:relativeY}
 Y_I(\lambda)=\left\{(x,\xi_0,\ldots,\xi_a):
 \begin{array}{l}
 x=(z,y)\in U(\lambda),\quad \xi_j\in B_{i_j}(\lambda),\\
 |z_\gamma-(f_{i_j})_\gamma(\xi_j)|<c(\lambda)\eta
       \quad\text{for every }j,\gamma
 \end{array}\right\}.
\end{equation}
Let \(\mathcal K_a\) consist of all ordered tuples of length \(a+1\),
including repetitions. For \(I\in\mathcal K_a\), \(0\le a\le J=n+1\), take
\(Y_I(\lambda_a)\) and \(Y_I(\lambda_{a+1})\) as the two tuple
domains of Section~\ref{fi:sc:composition}. The mixed face \(\delta\)
pulls back at level \(\lambda_{a+1}\) by forgetting one parameter.
Use the Euclidean form metric and volume in the product coordinates.
The bundle \((\rho\circ\pi_x)^*F_p\) has the pullback metric, where
\(\pi_x\) forgets all parameters. The maximal Dolbeault differential
acts in all variables.

\begin{lemma}[Parameter integration]\label{fi:it:lem:relative}\label{app:parameter}
On the tubes \eqref{fi:it:eq:relativeY}, the pullback
\(\iota u=(\pi_x^*u)_i\) and the partition map
\eqref{fi:it:eq:integral-assembly} satisfy \eqref{fi:sc:transfer-data}
on maximal domains through degree \(n\), with polynomial \(L^2\)
norm bounds through degree \(n+1\).
\end{lemma}
\begin{proof}
Set \(\chi_i(\xi,z)=\alpha_i(\xi)\psi((z-f_i(\xi))/\eta)\).
Take nonnegative cutoffs with \(\alpha_i=1\) on a fixed neighborhood
of \(B_i^-\) and \(\psi=1\) on a fixed neighborhood of the closed
polydisc of radius \(1/20\). In normalized coordinates their supports
lie a fixed distance inside the parameter and tube bounds at every
level. Extend \(\chi_i\) by zero and put
\(\Xi(z)=\sum_i\int_{B_i^+}\chi_i(\xi,z)\,\dd V_\xi\), using unit
mass in dimension zero. By \eqref{fi:it:eq:relativecover} there are
\(i\) and \(\xi_*\in B_i^-\) with \(|z-f_i(\xi_*)|_\infty<\eta/20\).
Cauchy's estimate on the fixed enlargement bounds \(Df_i\) uniformly.
Thus \(\chi_i=1\) on \(|\xi-\xi_*|<c\eta\), and with
\(d_*=\max_i\dim_\C B_i^+\) we obtain
\begin{equation}\label{fi:it:eq:interior-radius}
 \Xi(z)\ge c'\eta^{2d_*},\qquad \eta=p^{-1}.
\end{equation}
The quotient rule and the polynomial bounds for \(K_p\) and the
derivatives of the cutoffs give polynomial derivative bounds for
\(\varphi_i=\chi_i/\Xi\).

Write \(u_{\rm hor}\) for the component containing only \(d\bar x\)
differentials. With \(\db_x\) acting in \(x\), set
\begin{align}
 \pi_a u={}&\sum_{i_0,\ldots,i_a}\int
 \varphi_{i_0}(\xi_0,z)\,
 \db_x\varphi_{i_1}(\xi_1,z)\wedge\cdots\wedge
 \db_x\varphi_{i_a}(\xi_a,z)\notag\\
 &\hspace{2em}\wedge (u_I)_{\rm hor}(x,\xi_0,\ldots,\xi_a)
 \,\dd V_{\xi_0}\cdots\dd V_{\xi_a}.
 \label{fi:it:eq:integral-assembly}
\end{align}
For \(x\in U^-\) extend the integrand by zero from the fiber of
\(Y_I(\lambda_{a+1})\) to the product of parameter polydiscs.
The two face pullbacks of a source form agree, and normalization
gives \(\pi\cR\iota=R\). In the ordinary-cover calculation
following \eqref{fi:sc:partition-map}, replace summation over an
omitted index by summation and integration over its parameter.
The first face then contributes \(\sum_i\int\varphi_i=1\) and
each other face contributes \(\sum_i\int\db_x\varphi_i=0\).
Over any compactly supported test form on \(U^-\) the integrand is
supported inside the tuple domain. Since
\((\db u)_{\rm hor}=\db_x(u_{\rm hor})\), the distributional product
rule proves \(\db_-\pi=\pi D_-\) on the maximal domain and shows that
\(\pi\) preserves that domain through degree \(n\).

Each face pullback adds one parameter ranging in a polydisc of
bounded volume, so Fubini bounds its \(L^2\) norm. The same argument
gives \(\|\iota\|\le C\sqrt{K_p}\) in product coordinates. For
\(\pi\), apply Cauchy--Schwarz on each fiber. The bundle metric is
constant there, and the number of tuples and cutoff derivatives have
polynomial bounds. The source norm comparisons therefore give the
claimed bounds through degree \(n+1\).
\end{proof}

\subsection{Counting restrictions on the parameter tubes}\label{fi:new:sec:degrees}
If each tuple contributes at most \(N_{\rm tuple}(p)\) holomorphic
restrictions, \eqref{fi:new:eq:restriction-sum} bounds the total by
\(N_{\rm tuple}(p)\sum_{a=0}^n K_p^{a+1}\).
By \eqref{fi:it:eq:relativecover} this costs at most a fixed power
of \(1+\log p\) times \(N_{\rm tuple}(p)\). Each rescaling uses a
parameter cover followed by an ordinary subdivision. Adding the two
initial covers and the final target cover gives at most \(2k+3\).
Appendix~\ref{app:domain-levels} gives uniform bounds on the number
of variables and the boundary margins.

\section{Metric bounds on the local domains}\label{fi:it:sec:metric}

Take the two initial subdivisions of Proposition~\ref{prop:one-scale},
now with
\begin{equation}\label{fi:it:eq:scales}
 t=p^{-1/(2k+5)},\qquad P_k=pt^{2k+4}=t^{-1}.
\end{equation}
At rank zero \eqref{fi:it:eq:physical-remainder} already gives
a bounded weight for \(L^p\). At a center of rank \(m>0\), the initial
coordinates are given by \(S_0(t,u,v)\) in \eqref{fi:it:eq:initial}.
We shall use \eqref{fi:it:eq:final-expansion} to bound the weight
away from the graph, throughout each tube.

\subsection{Parameter tubes}

At every stage \(0\le\ell<k\), choose the tube width \(\eta\)
and the accuracy \(\epsilon\) of the product maps by
\begin{equation}\label{fi:it:eq:common-accuracy}
 \eta=p^{-1},\qquad \epsilon=\eta/40 .
\end{equation}
Take \(\mathcal A\) in \eqref{fi:it:eq:relativecover} inside the
smaller box of Proposition~\ref{fi:cell:prop:targetproducts}. Its
points lie within \(\epsilon\) of the images of the smaller parameter
products, and \(K_p\le C_k(1+\log p)^{\lambda_k}\) with
\(\lambda_k\) independent of \(p\).
For each map \(f_i(b,V)\), adjoin \(W\in\C^m\) and set
\[
 \xi=(b,V,W),\qquad \widehat f_{i,t}(\xi)=f_i(b,V)+t(W,0).
\]
For any
\(z\in\mathcal A\), taking \(W=0\) in an approximating product gives a
parameter \(\xi_*\) with \(|z-\widehat f_{i,t}(\xi_*)|\le\epsilon<\eta/20\).
On the fixed enlarged products, \(f_i\) takes values in a fixed
polydisc and \(W\) is bounded. Thus \(\widehat f_{i,t}\) is uniformly
bounded there, and Lemma~\ref{fi:it:lem:relative} applies.

\paragraph{Rescaling the first parameter.}
For \(I=(i_0,\ldots,i_a)\), the next subdivision restricts from
\(Y_I(\lambda_a)\) to \(Y_I(\lambda_{a+1})\). Subdivide the
\(b\)-coordinates of \(\xi_0\) alone at mesh \(t\). Take the outer
domain of \(\xi_0\) between the smaller approximating product and the
intermediate product of Lemma~\ref{fi:it:lem:renormalize}. Its
\(b\)-projection lies a fixed distance inside \(D_b^-\), so the mesh
centers \(b_*\) may be chosen in \(D_b^-\). In this copy set
\[
 (b,V,W)=(b_*+tU,V,W),\qquad \zeta=(U,V,W).
\]
If \(m_i=\dim b\) on chart \(i=i_0\), then
\(\dd V_{\xi_0}=t^{2m_i}\dd V_\zeta\). Both the affine coordinate
change and its inverse have polynomial bounds in \(p\).
Lemma~\ref{fi:it:lem:renormalize} gives the next expansion, with \(W\)
as its transverse variable. All other parameter copies remain at unit
scale. Since \(k\) is fixed, the successive enlarged products can be
chosen with margins independent of \(p\).

\subsection{The weight on a parameter tube}

Let \(z_0=(u,v)\) be the normalized physical coordinates on the initial
box. At stage \(j\), let \(z_{j+1}\) be the rescaled coordinates in
the first parameter of the chosen tuple. The defining inequalities
of the tube give \(|z_j-\Psi_{j,t}(z_{j+1})|\le C_k\eta\). Both
points lie in the fixed parameter polydisc on which
\(\|D_z S_j(t,\cdot)\|\le C_k t\) by \eqref{fi:it:eq:physical-affine}
and Cauchy's inequalities. Apply the mean value estimate along the
segment in this polydisc. Since \(S_{j+1}(t,\cdot)=S_j(t,\Psi_{j,t})\),
summing gives
\begin{equation}\label{fi:it:eq:errorrecurrence}
 |S_0(t,z_0)-S_k(t,z_k)|\le C_k t\eta .
\end{equation}
The gradient of the original smooth weight \(\phi\) is bounded on
the physical chart. Thus \eqref{fi:it:eq:final-expansion}, with its
holomorphic correction \(g(z_k)\), gives throughout the final tube
\begin{equation}\label{fi:it:eq:weight-decomposition}
 p\phi(S_0(t,z_0))-2p\Rea g(z_k)
 =P_k Q_k(z_k)^*\boldsymbol\Sigma_k Q_k(z_k)
     +O_k(pt^{2k+5}+pt\eta).
\end{equation}
Here \(pt^{2k+5}=1\) and \(\eta=p^{-1}\). Only the scalar weight
has been compared with its value on the graph. The variables in
each parameter copy are still independent in the integrals of
Section~\ref{fi:it:sec:relative}.

\section{Counting holomorphic restrictions}\label{fi:it:sec:count}

For each cover we count intersecting tuples of a fixed length.
Table~\ref{tab:restriction-count} gives the count on \(X\) in its
first row and the counts over one center of curvature rank \(m\) in
the remaining rows. Here \(x\) denotes the transverse variables, and
\(b\) the parameters through which the leading map factors on the
chosen product chart.

\begin{table}[!htb]
\centering
\caption{Coordinate subdivisions and their multiplicities.}
\label{tab:restriction-count}
\small
\begin{tabular}{llll}
\hline
Localization & Variables & Mesh & Multiplicity\\
\hline
Outer atlas & all $n$ coordinates & $t$ & $O(t^{-2n})$\\
Initial parabolic & $x\in\C^m$ in a width-$t$ box & $t^2$ & $O(t^{-2m})$\\
Step \(\ell\) & product charts, tuple length \(a+1\) & & \(O(K_p^{a+1})\)\\
Step \(\ell\) & \(b\), \(\dim b\le m\)
 & \(t\) & \(O(t^{-2\dim b})\)\\
Final & range of \(Q_k\) in \(\C^m\) & $P_k^{-1/2}$ & $O(P_k^m)$\\
\hline
\end{tabular}
\end{table}
The last row follows by pulling back a cover of the range of \(Q_k\)
by polydiscs of width \(P_k^{-1/2}\), as in Step~2 of
Proposition~\ref{prop:one-scale}.

At a parameter cover, the number of chart tuples contributing to
\eqref{fi:new:eq:restriction-sum} is at most \(\sum_{a=0}^n K_p^{a+1}\).
Thus for fixed \(k\) the chart choices in the \(k\) parameter covers
contribute at most a fixed power of \(1+\log p\). Since \(\dim b\le m\),
each subdivision of the first \(b\)-block costs \(O(t^{-2m})\) tuples
by bounded overlap. Combining these \(k\) subdivisions with the initial
\(x\)-subdivision and the final target cover gives, per outer tuple and up to a
constant and the logarithmic factor,
\begin{equation}\label{fi:it:eq:telescope}
 (t^{-2m})^{k+1}P_k^m
 =t^{-2m(k+1)}(pt^{4+2k})^m
 =(pt^2)^m .
\end{equation}
At a rank-zero center the weight is already bounded by
\eqref{fi:it:eq:physical-remainder}, so the factor is one. Summing
over the \(O(t^{-2n})\) outer tuples and using \(m\le r\) and
\(pt^2\ge1\) gives, for some \(M_k\) independent of \(p\),
\begin{equation}\label{fi:it:eq:global}
 \rescount(p)\le C_k(1+\log p)^{M_k}t^{-2n}(pt^2)^r
 =C_k p^{\,r+2(n-r)/(2k+5)}(1+\log p)^{M_k}.
\end{equation}
Thus the factor \(t^{-2(n-r)}\) from the remaining \(n-r\)
directions comes entirely from the outer cover. All constants are
uniform in the centers.

\begin{proposition}[Global restriction identity]\label{fi:it:prop:assembly}
Fix \(k\ge1\). For all sufficiently large \(p\) there exist a
degree-\((-1)\) homotopy \(T\) and \(\rescount(p)\) holomorphic
restrictions \(\restr_\iota\) with input maps \(\inmap_\iota\) and
output maps \(\outmap_\iota\) such that, on the global maximal
Dolbeault domains of \(L^p\otimes E\),
\begin{equation}\label{fi:it:eq:assembledidentity}
 \operatorname{id}-\sum_{\iota=1}^{\rescount(p)}\outmap_\iota\restr_\iota \inmap_\iota
       =\db T+T\db.
\end{equation}
These maps satisfy the hypotheses of Proposition~\ref{fi:prop:localization},
with \(\rescount(p)\) bounded by \eqref{fi:it:eq:global}. 
\end{proposition}
\begin{proof}
\proofstep{1}{Verify the bounds on the final domains}
Each final domain retains the inequalities of the earlier covers in
their original variables. A parameter cover adds coordinate bounds
and the tube inequalities \eqref{fi:it:eq:relativeY}, normalized by
\(\eta\). Where a final target subdivision is needed, its new defining
functions are
\(P_k^{1/2}(Q_{k,a}-q_{*,a})\), \(1\le a\le m\), where \(q_*\)
is the center of the target polydisc.

For each tuple, all these functions extend to one polydisc.
Indeed, the atlas patches have closures in the original frame
charts. The centers of an intersecting tuple and its enlarged
physical box lie within \(C_kt\) of a common point. For small \(t\)
the whole box therefore lies in every frame chart of the tuple.
For each independent parameter block choose a radius strictly
between the outer level and the radius of holomorphic extension.
The product with the physical box contains the outer intersection
with a fixed relative margin. We estimate the defining functions on
this product and use the original intersections for the face maps.
All normalizations are powers of \(t^{-1}\), \(\eta^{-1}\), and
\(P_k^{1/2}\), so the bounds on this product are polynomial in \(p\).
Appendix~\ref{app:domain-levels} gives \eqref{eq:nested-domain-ratios}
at consecutive levels, and \eqref{fi:it:eq:interior-radius} gives the
cutoff bounds.

\proofstep{2}{Bound the weight in a holomorphic frame}
After the holomorphic change of frame given by \(pg(z_k)\),
\eqref{fi:it:eq:weight-decomposition} leaves
\(P_kQ_k^*\boldsymbol\Sigma_kQ_k+O_k(1)\) throughout the tube.
On the whole outer domain, \(|Q_k-q_*|\le CP_k^{-1/2}\).
Removing the pluriharmonic terms as in Step~2 of
Proposition~\ref{prop:one-scale} leaves
\(P_k(Q_k-q_*)^*\boldsymbol\Sigma_k(Q_k-q_*)=O(1)\) for the
quadratic term. At a rank-zero center we remove the pluriharmonic
Taylor polynomial. The remainder is \(O(pt^{2k+5})=O(1)\) by
\eqref{fi:it:eq:physical-remainder}. This proves \eqref{fi:eq:gaugebound}
on the outer domain in both cases. The fixed metric on \(E\) and
the polynomial bounds for coordinates and Jacobians then give
\eqref{fi:sc:normequivalence}.

\proofstep{3}{Compose the local identities}
On the final domains \eqref{fi:sc:local-homotopy} holds through
degree \(n\), with graph norm bounds for the homotopies through
degree \(n+1\). The holomorphic restrictions in these identities
satisfy Lemma~\ref{fi:lem:restriction} by Steps~1 and~2.
Compose the identities by Lemma~\ref{fi:lem:composition}, starting
at the final subdivision and using \eqref{fi:sc:composite-identity}
for the pair of levels assigned to each tuple. At the atlas
\(R=\operatorname{id}\), so the resulting identity is
\eqref{fi:it:eq:assembledidentity}. The lemma preserves polynomial
norm bounds and adds the restriction counts by
\eqref{fi:new:eq:restriction-sum}. The total is \eqref{fi:it:eq:global}.

\end{proof}

\subsection{The quantitative bound and the proof of Theorem~\ref{thm:main}}

Fix a nonzero coherent analytic sheaf \(\mathcal F\) on \(X\) and suppose
\(1\le r=\crk(L)<n\). For each integer \(k\ge0\) there exist
constants \(C_k,N_k\), independent of \(p,q\), such that
\begin{equation}\label{fi:eq:iteration}
 h^q(X,\mathcal F\otimes L^p)
 \le C_k\,p^{\,r+2(n-r)/(2k+5)}(1+\log p)^{N_k}
 \qquad(0\le q\le n,\ p\ge1).
\end{equation}

\begin{proof}[Proof of Theorem~\ref{thm:main}]
\proofstep{1}{Estimate the groups with vector bundle coefficients}
Choose \(h\) with maximum curvature rank \(r=\crk(L)\). For
\(0<r<n\) and vector bundle coefficients, Proposition~\ref{prop:one-scale}
gives \eqref{fi:eq:iteration} when \(k=0\). For fixed \(k\ge1\),
Proposition~\ref{fi:it:prop:assembly} gives the global identity and
its restriction count \eqref{fi:it:eq:global}. Applying
Proposition~\ref{fi:prop:localization} gives \eqref{fi:eq:iteration}
by \eqref{fi:eq:localization}.

At the two extreme ranks, Demailly's estimate for a fixed coherent
sheaf of support dimension \(e\) suffices
\cite{Demailly2010}*{Lemma~2.1}. It gives \(O(p^e)\) for \(r=n\).
When \(r=0\), the chosen metric has zero curvature, hence
\(c_1^{\rm BC}(L^p)=0\) for every \(p\), and the estimate gives \(O(1)\).

\proofstep{2}{Pass to coherent coefficients and choose \(k\)}
Let \(d=\dim\operatorname{Supp}\mathcal F\). The groups vanish for
\(q>d\). For \(0\le q\le d\), consider a summand of
\eqref{eq:support-comparison} and put \(e=\dim Y_j\),
\(b=\crk(f_j^*L)\). By \eqref{eq:pullback-rank},
\(b\le\max_{Y_j}\rk(f_j^*\Theta_{L,h})\le\min\{r,e\}\).
Together with \(e\le d\le n\), this gives
\[
 b+\frac{2(e-b)}{2k+5}\le r+\frac{2(n-r)}{2k+5}.
\]
For \(0<b<e\) apply Step~1 on \(Y_j\) with coefficients \(V_j\).
For \(b=0\) use the bounded estimate, and for \(b=e\) use \(O(p^e)\).
Taking the largest logarithmic power among the finitely many terms
and degrees gives \eqref{fi:eq:iteration}.

Given \(\varepsilon>0\) and \(0<r<n\), fix \(k\) large enough that
\(\delta=\varepsilon-2(n-r)/(2k+5)>0\). The logarithmic factor in
\eqref{fi:eq:iteration} is \(O(p^\delta)\), which proves
\eqref{eq:near-endpoint}. In degrees \(q=0,d\), the sharper exponent
\(\min\{r,d\}\) follows from Proposition~\ref{cor:section-growth}.
\end{proof}

\begin{remark}[The slowly increasing factor]\label{rem:slow-factor}
Fix \(1\le s\le n\) as in the conjecture of
Section~\ref{sec:introduction} and put
\(\gamma_p=p^{-(n-s)}\max_{0\le q\le n}h^q(X,L^p\otimes E)\).
If \(\gamma_p\) is bounded, then for every \(a(p)\to\infty\) the
conjectural inequality holds eventually in all degrees. Conversely,
if \(\gamma_p\) is unbounded, take \(p_j\uparrow\infty\) with
\(\gamma_{p_j}>(j+1)^s\). Set \(a(p)=1\) for \(p<p_1\) and
\(a(p)=j\) for \(p_j\le p<p_{j+1}\). The inequality then fails at
each \(p_j\). In \eqref{fi:eq:iteration}, the constants \(C_k,N_k\)
depend on \(k\). Choosing \(k\) before letting \(p\to\infty\)
gives an arbitrarily small loss in the exponent.
\end{remark}

On a compact complex surface with \(\crk(L)=1\), one obtains
\(h^1(X,\mathcal F\otimes L^p)=O(p)\) for every fixed coherent sheaf
\(\mathcal F\). The group vanishes if the support is zero-dimensional.
Otherwise apply \eqref{eq:support-comparison} and put \(M_j=f_j^*L\).
When \(\dim Y_j=2\), the pullback curvature has rank at most one, so
\(c_1(M_j)^2=0\). The Riemann--Roch formula
\cite{AS1968}*{\S4, (4.3)}, \cite{MaMarinescu2007}*{Theorem~1.4.6, (1.4.26)}
therefore gives
\(\chi(Y_j,V_j\otimes M_j^p)=O(p)\). The endpoint bounds and
\(h^1=h^0+h^2-\chi\) give \(O(p)\) in degree one as well.
When \(\dim Y_j\le1\), every summand is already in an endpoint degree.

\begin{example}[A rank jump in dimension three]\label{ex:rank-jump}
Let \(X\) be the blowup of \(\mathbb P^1\times\mathbb P^2\) at a point.
Let \(f:X\to\mathbb P^1\) be the blowup map followed by projection
onto the first factor. In local product coordinates centered at the
point being blown up, one chart of the blowup map is
\((u,v,w)\mapsto(uv,v,vw)\). In this chart \(f(u,v,w)=uv\).
Equip \(L=f^*\mathcal O_{\mathbb P^1}(-1)\) with the pullback of the
tautological metric induced by the standard Hermitian norm on \(\C^2\).
Where \(df\ne0\) the curvature kernel is \(\ker df\). Since
\(df=v\,du+u\,dv\), the kernel equals \(\ker du\) along
\(u=w=0,\ v\ne0\), and \(\ker dv\) along \(v=w=0,\ u\ne0\).
These planes have different limits at the origin, so they cannot be
the tangent spaces of a smooth foliation near the origin.

Take \(y\in\mathbb P^2\) distinct from the second coordinate of the
blowup center. The strict transform of \(\mathbb P^1\times\{y\}\)
is a section of \(f\), and \(L\) has degree \(-1\) on it.
Thus \(L\) is not nef, its real Chern class is nonzero, and the
pullback metric of rank at most one gives \(\crk(L)=1\).

For the blowup, the structure sheaf has direct image
\(\mathcal O_{\mathbb P^1\times\mathbb P^2}\) and vanishing higher
direct images
\cite{DPS1996}*{proof of Proposition~1.6, with exceptional divisor coefficient zero}.
The projection onto \(\mathbb P^1\) has the same properties since
\(H^0(\mathbb P^2,\mathcal O)=\C\) and \(H^j(\mathbb P^2,\mathcal O)=0\)
for \(j>0\). By the projection formula these higher direct images
still vanish after twisting by the pullback of
\(\mathcal O_{\mathbb P^1}(-p)\). The Leray spectral sequences for the
two maps therefore give
\[
 h^1(X,L^p)=h^1(\mathbb P^1,\mathcal O(-p))=p-1,\qquad p\ge1 .
\]
Thus the exponent \(r=1\) is attained in degree one.
\end{example}

\section{Nef line bundles on manifolds in Fujiki's class}
\label{sec:nef-proof}

Suppose first that \(X\) has a K\"ahler form \(\omega\).
Since \(L\) is nef, for each \(p\ge1\) we can choose a metric on
\(L^p\) with \(\Theta_{L^p}\ge-\omega\). We use this lower bound
to estimate the dimension of the space of harmonic forms by an index.
The index depends on an auxiliary integer \(m\). Taking \(m\) fixed
gives the exponent \(\nu(L)\), while \(m=p\) gives \(n-q\).

\subsection{Sections of the auxiliary line bundle}

For each sufficiently large integer \(m\), we construct a smooth Hermitian
line bundle \(A_m\) with unitary connection \(\nabla\). Its curvature
\(R^{A_m}=\nabla^2\) satisfies
\[
 iR^{A_m}=m\omega+\eta_m,\qquad \|\eta_m\|_{C^0}\le C ,
\]
where \(\eta_m\) is a smooth real two-form and \(C\) is independent of \(m\).
We choose smooth sections \(s_1,\ldots,s_N\) of \(A_m\), each supported
in a ball of radius at most \(C/\sqrt m\). For \(s=(s_1,\ldots,s_N)\),
with the direct sum metric, the required pointwise estimates are
\begin{equation}\label{nef:tuple}
 1\le |s|^2\le C,\qquad |\nabla^{0,1}s|^2\le m/8 .
\end{equation}
Here \(|s|^2=\sum_i|s_i|^2\) and
\(|\nabla^{0,1}s|^2=\sum_i|\nabla^{0,1}s_i|^2\), where \(\nabla^{0,1}\)
is the \((0,1)\) part of \(\nabla\). The number of sections may depend on \(m\).

\begin{proof}[Construction of \(A_m\) and its sections]
\proofstep{1}{Choose the connection}
Round the coordinates of \(m[\omega]/(2\pi)\) in a basis of the
integral lattice in \(H^2(X,\mathbb R)\). Represent the rounding
error by \(\eta_m/(2\pi)\), using a bounded linear combination of
fixed closed forms. The class \([(m\omega+\eta_m)/(2\pi)]\) is
integral. Choose \(A_m\) with a unitary connection of curvature
\(-i(m\omega+\eta_m)\) \cite{Popovici2013}*{\S1.1, p.~315}.

\proofstep{2}{Use one Gaussian section on each ball}
Fix \(a\ge1\) and set \(\rho=a/\sqrt m\). Choose centered
holomorphic coordinates with
\(\omega(0)=(i/2)\sum_j dz_j\wedge d\bar z_j\). We can cover \(X\)
by balls \(\{|z|<\rho/4\}\) whose enlargements \(\{|z|<\rho\}\)
have bounded overlap, by choosing separated centers in a finite
atlas. On each enlarged ball compare \(\nabla\) with the unitary
connection having one-form \(\tfrac m4(\bar\partial-\partial)|z|^2\).
The Gaussian \(u(z)=e^{-m|z|^2/4}\) satisfies
\(\bar\partial u+\tfrac m4\bar\partial |z|^2\,u=0\)
\cite{MaMarinescu2007}*{(7.1.1)--(7.1.2)}.
In a unitary frame let \(\alpha\) denote the difference of the two
connection forms. Their curvatures differ by \(d\alpha=O(1+m\rho)\)
on the ball. Integration along radii gives a primitive \(\gamma\) of
\(d\alpha\) with \(|\gamma|\le C\rho(1+m\rho)=O_a(1)\). The imaginary
one-form \(\alpha-\gamma\) is exact on the ball. A unitary change
of frame therefore reduces \(\alpha\) to \(\gamma\), giving
\(|\nabla^{0,1}u|\le C_a|u|\).

Choose a cutoff \(0\le\chi_i\le1\) equal to one on \(\{|z|\le\rho/2\}\),
supported in \(\{|z|<\rho\}\), with \(|d\chi_i|\le C/\rho\).
Extend \(\chi_i u\) by zero to a section \(s_i\). The inner balls
give \(|s|^2\ge c_a:=e^{-a^2/32}\), and bounded overlap gives
\(|s|^2\le C\). On the support of \(d\chi_i\) we have
\(|u|^2\le e^{-a^2/8}\). Hence the product rule gives
\[
 |\nabla^{0,1}s|^2\le Cma^{-2}e^{-a^2/8}+C_a .
\]
Choose \(a\) with \(Ca^{-2}e^{-a^2/8}\le c_a/16\), then choose
\(m\) with \(C_a\le c_a m/16\). Multiplying the sections by
\(c_a^{-1/2}\) gives \eqref{nef:tuple}.
\end{proof}

\subsection{Energy and local mass of harmonic forms}

Fix \(0\le q\le n\) and a Hermitian holomorphic vector bundle \(E\).
Let \(P\) be a holomorphic line bundle with a smooth metric satisfying
\(\Theta_P\ge-\omega\). Give
\(G=E\otimes K_X^{-1}\) the induced metric and use the volume form
\(dV_\omega=\omega^n/n!\). Hodge theory identifies \(H^q(X,E\otimes P)\)
with the harmonic \(G\otimes P\)-valued \((n,q)\)-forms \(w\)
\cite{MaMarinescu2007}*{Theorem~1.4.1 and (1.4.56)}.
Put \(V=\Omega_X^{n-q}\otimes G\), and let \(\mathcal H\) be the
space of sections \(v=*w\) of \(V\otimes P\), where \(*\) is the
complex-linear Hodge star on the form factor. It preserves norms.
Write \(D'\) for the \((1,0)\) exterior Chern derivative. Since
\(w\) is harmonic and \(D'w=0\) by degree, the Bochner--Kodaira--Nakano
identity bounds \(\|(D')^*w\|^2\) by \(C\|w\|^2\). Indeed the line curvature
contributes at least \(-q\), and the curvature of \(G\) is fixed
\cite{MaMarinescu2007}*{Theorem~1.4.12}. The adjoint identity
\((D')^*=\pm *\bar\partial *\) therefore gives
\begin{equation}\label{nef:energy}
 \|\bar\partial v\|^2\le C\|v\|^2,\qquad v\in\mathcal H .
\end{equation}
The constant in \eqref{nef:energy} is uniform over \(P\) and its
metric whenever \(\Theta_P\ge-\omega\). We also need an estimate
on the metric balls \(B(x,t)\) of radius \(t\) about \(x\).
It holds for any Hermitian form \(\omega\).

\begin{lemma}[Local mass]\label{nef:mass-lemma}
On a compact Hermitian manifold \((X,\omega)\), fix a Hermitian
holomorphic vector bundle \(G\) and an integer \(0\le q\le n\).
If a smoothly metrized holomorphic line bundle \(P\) satisfies
\(\Theta_P\ge-\omega\), every harmonic \(G\otimes P\)-valued
\((n,q)\)-form \(w\) satisfies
\begin{equation}\label{nef:local-mass}
 \int_{B(x,t)}|w|^2\,dV_\omega\le Ct^{2q}\|w\|^2,
 \qquad x\in X,\quad 0<t<t_0 .
\end{equation}
The constants \(C,t_0>0\) depend only on \(X,\omega,G,q\).
\end{lemma}
\begin{proof}
For \(q=0\), the integral over the ball is at most \(\|w\|^2\). For \(q>0\),
center holomorphic coordinates at \(x\) with
\(\omega(x)=(i/2)\sum_j dz_j\wedge d\bar z_j\), and put
\(\sigma(t)=\int_{|z|<t}|w|^2dV_\omega\).
Write \(\omega_j=\omega^j/j!\). Let \(T\) be the \((n-q,n-q)\)-form
associated to \(w\) in \cite{WangZhou2019}*{Proposition~3.1},
normalized by \(T\wedge\omega_q=|w|^2dV_\omega\). Since \(w\) is
harmonic and the line curvature contributes at least \(-q|w|^2\),
that proposition gives
\[
 i\partial\bar\partial T\wedge\omega_{q-1}\ge-C|w|^2dV_\omega .
\]
Following Berndtsson's proof of the local mass estimate
\cite{Berndtsson2002}*{Theorem~2.1, (2.4)--(2.6), pp.~301--303},
we test this inequality with \(\chi=(t^2-|z|^2)_+\) and integrate by parts.
The terms in which a derivative falls on \(\omega_{q-1}\) are
\(O(t)\sigma(t)\), since \(|d\chi|\le2t\) and \(\chi\le t^2\). The constant form \(i\partial\bar\partial|z|^2\) agrees with
\(2\omega\) at \(x\). Its interior contribution is
\((2q+O(t))\sigma(t)\), and the radial boundary term is at most
\((1+Ct)t\sigma'(t)\). Thus, after increasing \(C\),
\[
 t\sigma'(t)\ge(2q-Ct)\sigma(t).
\]
It follows that \(e^{Ct}t^{-2q}\sigma(t)\) is nondecreasing.
For a fixed outer radius \(R\) and \(0<t<R\), this gives
\[
 \sigma(t)\le e^{C(R-t)}(t/R)^{2q}\sigma(R)
 \le e^{CR}(t/R)^{2q}\|w\|^2.
\]
Using a finite atlas, we may take \(R\) and \(C\) independent of \(x\).
In these coordinates \(B(x,t)\subset\{|z|<C't\}\) for all sufficiently small \(t\).
This proves \eqref{nef:local-mass}.
\end{proof}

\subsection{The spectral gap and the trace estimate}

\begin{proof}[Proof of Theorem~\ref{thm:nef}]
\proofstep{1}{Estimate the spectrum}
Give \(W=V\otimes P\otimes A_m\) the product connection, with
\((0,1)\) exterior derivative \(D''\). On \(W\)-valued \((0,*)\)-forms
the operator \(\mathscr D=\sqrt2(D''+(D'')^*)\) is elliptic, self-adjoint
and odd. Use the parallel identification of these forms with
\(K_X^{-1}\otimes W\)-valued \((n,*)\)-forms. Let \(R\) be the curvature
of the latter coefficient bundle, \(\Delta'\) the Laplacian of its
\((1,0)\) exterior derivative, and \(\Lambda_\omega\) the adjoint of
exterior multiplication by \(\omega\). For an odd form \(u\), the
curvature identity \cite{Le2010}*{(3.13.3), pp.~53--54} and
\((D'')^2=R^{0,2}\wedge\) give
\[
 \tfrac12\|\mathscr Du\|^2
 =\langle\Delta' u,u\rangle
  +\langle[iR^{1,1}\wedge,\Lambda_\omega]u,u\rangle
  +2\Rea\langle R^{0,2}\wedge u,u\rangle
 \ge (m-C)\|u\|^2 .
\]
Indeed, in bidegree \((n,j)\) the term \(m\omega\) contributes \(jm\)
and \(\Theta_P\ge-\omega\) contributes at least \(-j\). The remaining
curvature terms, including \(R^{0,2}\), are uniformly bounded.
On each positive eigenspace of \(\mathscr D^2\), \(\mathscr D\) identifies
the two parities \cite{MaMarinescu2007}*{(1.5.35)}. Thus
\(\mathscr D^2\ge2m-C\) on \((\ker\mathscr D)^\perp\), uniformly in \(P\).

\proofstep{2}{Project onto the kernel}
Let \(\Pi\) be orthogonal projection onto \(\ker\mathscr D\).
Each \(s_iv\), \(v\in\mathcal H\), has degree zero. The spectral gap,
the product rule, \eqref{nef:tuple} and \eqref{nef:energy} give
\[
 \sum_i\|s_iv-\Pi(s_iv)\|^2
 \le \frac{4}{2m-C}\sum_i
      \bigl(\|(\nabla^{0,1}s_i)v\|^2+\|s_i\bar\partial v\|^2\bigr)
 \le \frac{m/2+C'}{2m-C}\,\|v\|^2 .
\]
Choose \(m_0\) large enough that for \(m\ge m_0\) the fraction is at most
\(1/2\) and each support ball has radius less than \(t_0\).
Since \(\sum_i\|s_iv\|^2\ge\|v\|^2\), orthogonality gives
\[
 \tfrac12\|v\|^2\le\sum_i\|\Pi(s_iv)\|^2 .
\]

\proofstep{3}{Estimate the dimension}
Choose an orthonormal basis \((\zeta_k)\) of \(\ker\mathscr D\).
For its degree-zero component \((\zeta_k)_0\), contraction with \(s_i\)
gives \(s_i^*(\zeta_k)_0\), supported in a ball of radius \(C/\sqrt m\).
By Parseval the right-hand side of Step~2 is
\(\sum_{i,k}|\langle v,s_i^*(\zeta_k)_0\rangle|^2\).
For a linear functional \(\lambda\) on \(\mathcal H\), the quadratic
form \(v\mapsto|\lambda(v)|^2\) has trace \(\|\lambda\|^2\). Cauchy--Schwarz on the support ball and
\eqref{nef:local-mass} therefore give
\[
 \tfrac12\dim\mathcal H
 \le \sum_{i,k}\sup_{\substack{v\in\mathcal H\\\|v\|\le1}}
           |\langle v,s_i^*(\zeta_k)_0\rangle|^2
 \le C m^{-q}\sum_{i,k}\|s_i^*(\zeta_k)_0\|^2
 \le C m^{-q}\dim\ker\mathscr D .
\]
The last inequality uses \(\sum_i|s_i|^2\le C\).
By Step~1 the odd kernel vanishes. Thus \(\dim\ker\mathscr D\)
is the index of the even-to-odd part of \(\mathscr D\), whose
principal symbol is the twisted Dolbeault symbol. Write
\(\operatorname{ch}\) and \(\operatorname{td}\) for the Chern
character and Todd class. The Atiyah--Singer index theorem
\cite{AS1968}*{Theorem~(2.12) and \S4},
\cite{MaMarinescu2007}*{Theorem~1.3.9, (1.3.46), p.~29}
now gives
\begin{equation}\label{nef:trace-index}
 h^q(X,E\otimes P)\le C m^{-q}\int_X
 e^{c_1(P)+c_1(A_m)}\operatorname{ch}(V)\operatorname{td}(T^{1,0}X).
\end{equation}
The integral is taken in degree \(2n\). For fixed \(X,\omega,E,q\),
the same \(C\) and \(m_0\) apply to all \(m\ge m_0\) and all smooth
metrics on holomorphic line bundles \(P\) with \(\Theta_P\ge-\omega\).

\proofstep{4}{Choose the scale}
Take \(P=L^p\) with \(\Theta_P\ge-\omega\). For \(m=m_0\), expand
\(e^{p c_1(L)}\) in \eqref{nef:trace-index}. Since
\(c_1(L)^{\nu(L)+1}=0\), the index is \(O(p^{\nu(L)})\).
For \(m=p\ge m_0\) one has
\(c_1(A_p)=p[\omega]/(2\pi)+O(1)\) in real cohomology, so the
degree-\(2n\) part of the integrand is \(O(p^n)\). Multiplying by \(p^{-q}\) in \eqref{nef:trace-index} gives
\(O(p^{n-q})\). Taking the smaller of the two bounds, we find
\[
 h^q(X,E\otimes L^p)=O\bigl(p^{\min\{\nu(L),\,n-q\}}\bigr).
\]

\proofstep{5}{Pass to coherent coefficients}
For \(X\in\mathcal C\) choose the maps \(f_j:Y_j\to X\) in
\eqref{eq:support-comparison} with \(Y_j\) K\"ahler. The pullback
\(f_j^*L\) is nef, since the pullback of a Hermitian form on \(X\)
is bounded above by a multiple of a K\"ahler form on \(Y_j\).
Also \(c_1(f_j^*L)^a=f_j^*(c_1(L)^a)\), so \(\nu(f_j^*L)\le\nu(L)\).
Step~4 bounds the \(j\)-th summand by
\[
 O\bigl(p^{\min\{\nu(f_j^*L),\,\dim Y_j-q_j\}}\bigr).
\]
Since \(\dim Y_j-q_j\le d-q\), summing gives \eqref{eq:nef-main}.
\end{proof}

\section{Metrics with analytic singularities}\label{sec:analytic-poles}

We compute the multiplier ideals on a log resolution of \(h\).
For a component of the polar divisor with coefficient \(a\), the
integer \(\lfloor pa\rfloor\) gives its multiplicity in the divisor
of the multiplier ideal on the resolution. The fractional parts are
bounded by one, so their contribution to the smooth weight is uniformly
bounded in \(p\).

\begin{proof}[Proof of Theorem~\ref{thm:singular}]
\proofstep{1}{Identify the cohomology on a log resolution}
The case \(n=0\) is immediate. For \(n>0\), choose a finite
cover with weights \(c_j\log|f_j|^2+\psi_j\), where \(f_j\) is a
tuple of generators of a local ideal \(I_j\) and \(\psi_j\) is
smooth. To resolve all the ideals \(I_j\) on one manifold, choose an
integer \(p_0\) such that \(p_0c_j>n+1\) for every \(j\).
Write \(\mathcal I(u)\) for the multiplier ideal of a weight \(u\).
Skoda's division theorem gives, on the \(j\)-th chart,
\[
 \mathcal I(h^{\otimes p_0})
 =I_j\,\mathcal I\bigl((p_0c_j-1)\log|f_j|^2\bigr).
\]
Indeed, \(e^{-p_0\psi_j}\) is locally bounded above and below by
positive constants. We may append zeros to \(f_j\) to obtain at least
\(n+1\) generators, and apply \cite{Skoda1972}*{Theorem~1, (3.1)--(3.2), pp.~555--556}
with parameter \((p_0c_j-1)/n>1\).

The ideal \(\mathcal I(h^{\otimes p_0})\) is coherent
\cite{Nadel1990}, \cite{MaMarinescu2007}*{Theorem~2.3.15}.
Principalize this ideal by blowups with smooth centers
\cite{BierstoneMilman1997}*{Theorem~1.10}, as in
\cite{MaMarinescu2007}*{proof of Lemma~2.3.19, pp.~104--105}. If \(X\in\mathcal C\), first pull the ideal back to a compact K\"ahler
modification and perform the blowups there. Write \(\mu:Y\to X\) for
the resulting map. In the latter case \(Y\) is K\"ahler.

The ideal \(\mathcal I(h^{\otimes p_0})\mathcal O_Y\) is invertible
and defines a divisor with simple normal crossings. If a product of
ideals \(AB\) is invertible, then \(B(AB)^{-1}\) is an inverse for \(A\).
Applying this to Skoda's factorization shows that each
\(I_j\mathcal O_Y\) is locally principal over its chart.

If \(s\) generates \(I_j\mathcal O_Y\) locally, write \(f_j\circ\mu=s u_j\).
The entries of \(u_j\) have no common zero. Thus the pulled-back weight is
\[
 c_j\log|s|^2+c_j\log|u_j|^2+\psi_j\circ\mu ,
\]
with the last two terms smooth.
Let \(D_i\) be the irreducible components of the divisors defined by
the ideals \(I_j\mathcal O_Y\). Along \(D_i\) the logarithmic
coefficient is \(a_i=c_j\operatorname{ord}_{D_i}(I_j\mathcal O_Y)\).
It is independent of \(j\) on overlaps, since two local weights
of the same metric differ by a smooth function. For \(T=\mu^*L\), the Poincar\'e--Lelong
formula \cite{Lelong1962}*{p.~18},
\cite{MaMarinescu2007}*{Theorem~B.2.19, (B.2.18), p.~360} gives
\[
 c_1(T,\mu^*h)=\beta+\sum_i a_i[D_i],
\]
where \([D_i]\) is the current of integration and \(\beta\) is smooth
and closed. The form \(\beta\) is semipositive off the divisor,
hence everywhere by continuity. For \(p\ge1\), put
\[
 D_p=\sum_i\lfloor p a_i\rfloor D_i,\qquad M_p=T^p(-D_p),\qquad
 K_{Y/X}=K_Y\otimes\mu^*K_X^{-1}.
\]
At a normal crossing, integrability in each divisor coordinate gives
\(\mathcal I((\mu^*h)^{\otimes p})=\mathcal O_Y(-D_p)\)
\cite{MaMarinescu2007}*{Example~2.3.16(3), (2.3.21), p.~103}.
The change of variables for multiplier ideals
\cite{MaMarinescu2007}*{Lemma~2.3.23, (2.3.35)} gives
\[
 \mu_*\bigl(K_{Y/X}(-D_p)\bigr)=\mathcal I(h^{\otimes p}).
\]
The factor \(K_{Y/X}\) comes from the Jacobian of \(\mu\).
To apply relative vanishing, observe that \(\mu\) is locally K\"ahler.
For a composition of smooth blowups this follows on sufficiently small
coordinate balls from \cite{Fujiki1978}*{Lemma~4.4(1)--(2)}. When \(Y\)
is K\"ahler its open subsets already have K\"ahler metrics.
Moreover, \(\mu\) is bimeromorphic, so its generic relative dimension is zero. Matsumura's vanishing theorem
\cite{Matsumura2022}*{Remark~1.4(1) and Corollary~1.5, pp.~124--125}, applied to
\((T^p,(\mu^*h)^{\otimes p})\), gives
\(R^\ell\mu_*(K_Y\otimes T^p(-D_p))=0\) for \(\ell>0\).
Put \(G=\mu^*E\otimes K_{Y/X}\). The projection formula gives
\[
 R^\ell\mu_*(G\otimes M_p)=
 \begin{cases}
 E\otimes L^p\otimes\mathcal I(h^{\otimes p}),&\ell=0,\\
 0,&\ell>0.
 \end{cases}
\]
Thus only the zeroth row remains in the Leray spectral sequence for
\(\mu\). The edge isomorphism \cite{MaMarinescu2007}*{Theorem~2.3.24}
gives
\[
 H^q(X,E\otimes L^p\otimes\mathcal I(h^{\otimes p}))
 \simeq H^q(Y,G\otimes M_p).
\]

\proofstep{2}{Estimate the rounded line bundles}
Give \(\mathcal O_Y(D_i)\) a smooth metric \(k_i\), let \(s_i\) be its
canonical section and put \(\theta_i=c_1(\mathcal O_Y(D_i),k_i)\).
On \(T=\mu^*L\), the metric defined off the divisor by
\(h_0=\mu^*h\prod_i|s_i|_{k_i}^{2a_i}\) extends smoothly across it. Indeed, write the weight of \(\mu^*h\) as
\(\phi+\sum_i a_i\log|g_i|^2\) in local equations \(g_i\) of \(D_i\).
If \(\lambda_i\) is the corresponding weight of \(k_i\), the weight of
\(h_0\) is \(\phi+\sum_i a_i\lambda_i\). The tensor metric
\(h_0^p\otimes\bigotimes_i k_i^{-\lfloor pa_i\rfloor}\) on \(M_p\)
therefore has local weight
\[
 p\phi+\psi_p,\qquad \psi_p=\sum_i\{pa_i\}\lambda_i,\qquad
 \tfrac{i}{2\pi}\partial\bar\partial\phi=\beta ,
\]
where \(\{t\}=t-\lfloor t\rfloor\). All \(C^j\) norms of \(\psi_p\)
on a fixed finite cover of relatively compact frame domains are
uniformly bounded. The normalized curvature is
\(p\beta+\sum_i\{pa_i\}\theta_i\). Since \(\beta\ge0\), scaling a
fixed Hermitian form \(\omega_Y\) gives \(\Theta_{M_p}\ge-\omega_Y\)
for all \(p\). Choose \(\omega_Y\) K\"ahler when \(X\in\mathcal C\).

We adapt the rescaling in Berndtsson's proof of Theorem~2.3
\cite{Berndtsson2002}*{pp.~304--305} to the weights \(p\phi+\psi_p\) to obtain
\begin{equation}\label{eq:rounded-degree}
 h^q(Y,G\otimes M_p)=O(p^{n-q}).
\end{equation}
Represent the cohomology by harmonic \((n,q)\)-forms \(w\) with values
in \(G\otimes K_Y^{-1}\otimes M_p\). Center coordinates at \(y\), remove
the real affine Taylor polynomial of \(p\phi\) by a holomorphic change
of frame, and write \(z=tu\), with \(t=p^{-1/2}\) for large \(p\).
On the unit \(u\)-ball the weight and its first derivatives are
bounded by Taylor's formula. For \(j\ge2\) the \(j\)-th derivatives
of \(p\phi(tu)\) are \(O(pt^j)=O(1)\), and the derivatives of
\(\psi_p(tu)\) are uniformly bounded. If \(\beta=0\), then \(\phi\)
is pluriharmonic, so we can remove all of \(p\phi\) and take \(t\)
to be a fixed sufficiently small radius.
In both cases the rescaled Laplacians are uniformly elliptic with
uniformly bounded coefficients and derivatives. The interior estimate
\cite{Berndtsson2002}*{(2.9), p.~304} and \eqref{nef:local-mass} give
\[
 |w(y)|^2\le C t^{-2n}\int_{B(y,C't)}|w|^2\,dV_{\omega_Y}
            \le C t^{-2(n-q)}\|w\|^2 .
\]
For an orthonormal basis \((w_j)\) of harmonic forms, the fixed fiber
dimension and the pointwise bound give \(\sum_j|w_j(y)|^2\le C t^{-2(n-q)}\)
by \cite{Berndtsson2002}*{Lemma~4.1, (4.1)--(4.2), pp.~308--309}. Integration gives
\eqref{eq:rounded-degree} and, when \(\beta=0\), the bound \(O(1)\).

\proofstep{3}{Use the rank of the residual curvature}
Over the dense open subset of \(X\setminus Z_h\) where \(\mu\) is
biholomorphic, we have \(2\pi\beta=\mu^*\Theta_{L,h}\).
The locus of maximal rank of a smooth form is open. Thus the ranks
of these two forms have the same maximum,
\[
 \max_Y\rk\beta=\max_{X\setminus Z_h}\rk\Theta_{L,h}=b .
\]

When \(X\in\mathcal C\), our choice of \(Y\) in Step~1 makes it
K\"ahler. Apply \eqref{nef:trace-index} with \(P=M_p\),
\(m=m_0\), and coefficients \(G\). In real cohomology,
\[
 c_1(M_p)=p[\beta]+\sum_i\{pa_i\}\,c_1(\mathcal O_Y(D_i)).
\]
The coefficients \(\{pa_i\}\) lie in \([0,1)\), so
\(c_1(M_p)-p[\beta]\) is bounded in real cohomology. In the index integral
write \(e^{c_1(M_p)}=e^{p[\beta]}e^{c_1(M_p)-p[\beta]}\).
The second factor stays bounded, and \(\beta^{b+1}=0\) gives
\(e^{p[\beta]}=\sum_{j=0}^b p^j[\beta]^j/j!\).
Combining \(O(p^b)\) with \eqref{eq:rounded-degree} gives
\(O(p^{\min\{b,n-q\}})\) on \(Y\), hence on \(X\) by Step~1.

For general \(Y\), Step~2 gives \(O(1)\) when \(b=0\), and
\eqref{eq:rounded-degree} suffices when \(b\ge n-q\).

It remains to treat \(0<b<n-q\). Choose a finite cover of \(Y\) by
coordinate balls in the frame charts of Step~2. Use the fixed nested
radii of Section~\ref{fi:sc:composition}, so that each pair of
intersections satisfies \(U_I^-\Subset U_I^+\).
Each intersection is Stein, being the closed diagonal in the product
of the corresponding balls. In any one of these charts it is thus a
bounded pseudoconvex open set.

On \(U_I^+\), solve \(\db\) coefficientwise in positive degrees with
weight \(p\phi+|z|^2\) by \cite{Hor65}*{\S2.2}. Its Levi form is
bounded below by \(i\partial\bar\partial|z|^2\), and its difference
from \(p\phi+\psi_p\) is uniformly bounded. In the fixed frame of
\(G\) the norms of forms and the volume density are uniformly
comparable to the Euclidean ones. Thus the least-norm homotopies in
Appendix~\ref{fi:sc:contraction} have graph norms bounded independently
of \(p\). The face restrictions use the same bundle metric, and all
partition functions and intersections are fixed. The finite sums in
\eqref{fi:sc:additivefactorization} therefore give uniform bounds for
the operators on both sides of the holomorphic restrictions.
By \eqref{eq:localization-hs},
\[
 h^q(Y,G\otimes M_p)\le C\sum_I\|R_{I,p}\|_{\mathrm{HS}}^2 ,
\]
where \(R_{I,p}\) is scalar restriction from \(U_I^+\) to \(U_I^-\)
with weight \(p\phi\). Here we used the uniform bound on \(\psi_p\)
at both ends of the restriction and absorbed \(\rk G\) in \(C\).

Cover \(\overline{U_I^-}\) by finitely many fixed balls \(B\) for
which \(6B\Subset U_I^+\), where \(aB\) denotes the concentric ball
with radius multiplied by \(a\). On \(B\) the diagonal Bergman kernel
of \(U_I^+\) is bounded above by that of \(6B\), by restriction with
the same weight and volume. Integrating over the balls reduces the
estimate to the scalar restriction \(R_p\) from \(6B\) to \(B\).

\proofstep{4}{Estimate restriction on a fixed ball}
Translate and rescale to make \(B\) the unit ball, and
write \(dd^c=i\partial\bar\partial\), \(\omega_0=dd^c|z|^2/2\).
The weight on \(6B\) is \(p\phi\), with \(\phi\) smooth and
plurisubharmonic on a neighborhood of its closure. Let \(dV\)
be the fixed smooth volume there and put \(K=\sum_j|f_j|^2\)
for an orthonormal basis of the weighted Bergman space on \(6B\). We first show that
\[
 \|R_p\|_{\mathrm{HS}}^2=\int_B K e^{-p\phi}\,dV
 \le C_n\int_{2B}(\omega_0+dd^c\log K)^n .
\]
Take \(f_1\) to be a normalized constant and put \(F=(f_1,\ldots,f_N)\).
On the graph of \(z\mapsto[F(z)]\) choose a maximal
\(\delta\)-separated set over \(B\) in the Euclidean and Fubini--Study
product metric, for fixed small \(\delta>0\).
The disjoint balls of radius \(\delta/2\) project into \(2B\).
Each contains graph volume at least \(c_n\delta^{2n}\) by Lelong
monotonicity \cite{DemaillyCADG}*{Ch.~III, (5.7)--(5.8)}.
Indeed, in a centered unitary affine chart the metric is uniformly
comparable to the Euclidean metric on a ball of fixed radius,
independently of \(N\), and the graph has complex dimension \(n\).
The number of centers is therefore at most a constant times
\(\int_{2B}(\omega_0+dd^c\log|F|^2)^n\), by the pullback formula for
the Fubini--Study form. The balls of radius \(\delta\) cover the graph
over \(B\). In the ball centered over \(x\), the unit Bergman function
\(\sum_{j=1}^N\overline{f_j(x)}f_j(z)/|F(x)|\) has squared modulus
at least \(c|F(z)|^2\) by the projective distance formula.
Each ball thus contributes bounded weighted mass. Summing and letting
\(N\to\infty\) proves the inequality. The polarized kernels converge
locally normally, so \(|F|^2\to K\) locally in \(C^2\).

On \(4B\), point extension with weight \(p\phi\)
\cite{DemaillyKollar2001}*{Theorem~2.1, p.~532} gives \(Ke^{-p\phi}\ge c\).
On balls of radius \(1/p\) about these points \(p\phi\) has bounded
oscillation, so the submean inequality gives \(Ke^{-p\phi}\le Cp^{2n}\).
Set \(\ell=\log(2+p)\), \(v=\log K+|z|^2/2\), and
\(u=p\phi+|z|^2/2+A\ell(|z|^2-9)\). Since
\(|\log K-p\phi|\le C\ell\), a fixed large \(A\) gives \(u<v-1\)
on \(2B\) and \(u>v+1\) near \(\partial(4B)\).
Take a smooth convex regularization of \(\max(u,v)\), nondecreasing in each argument, which equals
the larger argument when \(|u-v|\ge1\). This gives a smooth
plurisubharmonic function \(w\), equal to \(v\) on \(2B\) and to
\(u\) near \(\partial(4B)\). Positivity and Stokes' formula give
\[
 \int_{2B}(dd^cv)^n\le\int_{4B}(dd^cw)^n=\int_{4B}(dd^cu)^n
 \le\int_{4B}\bigl(p\,dd^c\phi+C\ell\omega_0\bigr)^n .
\]
Here \(dd^c\phi=2\pi\beta\) and \(\beta^{b+1}=0\). Expansion
of the last integral gives \(\|R_p\|_{\mathrm{HS}}^2=O(p^b\ell^{n-b})\).
Summing over the fixed cover and using the isomorphism of Step~1 yields
\begin{equation}\label{eq:rounded-log}
 h^q\bigl(X,E\otimes L^p\otimes\mathcal I(h^{\otimes p})\bigr)
 =O\!\left(p^b(\log(2+p))^{n-b}\right).
\end{equation}
\end{proof}

For \(p\ge1\) write \(F_p=E\otimes L^p\) and
\(\mathcal I_p=\mathcal I(h^{\otimes p})\).
Since \(F_p\) is locally free and \(\mathcal I_p=\mathcal O_X\)
off \(Z_h\), the quotient \(F_p/(F_p\otimes\mathcal I_p)\) is
supported on \(Z_h\). Its cohomology vanishes above \(\dim Z_h\).
The long exact sequence therefore gives
\begin{equation}\label{eq:above-pole-surjection}
 H^q(X,F_p\otimes\mathcal I_p)\twoheadrightarrow H^q(X,F_p)
 \qquad(q>\dim Z_h),
\end{equation}
as in \cite{WangZhou2019}*{\S7, (17)--(19)}. Thus
Theorem~\ref{thm:singular} also bounds \(h^q(X,F_p)\) in these
degrees. If \(Z_h\) is empty,
then \(\mathcal I_p=\mathcal O_X\) for every \(p\).

\begin{example}[A skyscraper sheaf at a pole]\label{ex:point-multiplier}
On \(X=\mathbb P^2\) with homogeneous coordinates \([Z_0:Z_1:Z_2]\),
put \(L=\mathcal O_{\mathbb P^2}(1)\) and \(x=[1:0:0]\). The two
sections \(Z_1,Z_2\) define a metric \(h\) on \(L\): in a frame \(e\),
write \(Z_j=f_j e\) and put \(|e|_h^{-2}=|f_1|^2+|f_2|^2\).
Its logarithmic weights are plurisubharmonic, with a single pole at
\(x\). In the frame \(Z_0\), with \(z_j=Z_j/Z_0\), the weight is
\(\log(|z_1|^2+|z_2|^2)\).
Let a nonzero germ at \(x\) have order \(k\). The squared modulus
of its leading homogeneous term has positive integral on the unit
sphere. Thus in polar coordinates the integrability condition is
\[
 \int_0^\delta r^{2k-2p+3}\,dr<\infty
 \quad\Longleftrightarrow\quad k\ge p-1 .
\]
Thus \(\mathcal I(h^{\otimes p})_x=\mathfrak m^{p-1}\), where
\(\mathfrak m\) is the maximal ideal at \(x\).
For the skyscraper sheaf \(\mathbb C_x\), the stalk of its tensor
product with this ideal is \(\mathfrak m^{p-1}/\mathfrak m^p\).
The monomials of degree \(p-1\) give a basis, so
\[
 h^0(X,\mathbb C_x\otimes L^p\otimes\mathcal I(h^{\otimes p}))
 =\dim_{\mathbb C}(\mathfrak m^{p-1}/\mathfrak m^p)=p.
\]
Here \(d=q=0\), so an estimate \(O(p^{d-q})\) would give \(O(1)\).
Off \(x\) the curvature is the pullback of the Fubini--Study form by
\([Z_0:Z_1:Z_2]\mapsto[Z_1:Z_2]\), and hence \(b=1\).
\end{example}

\appendix

\section{Restriction maps and nested domains}\label{app:restrictions}
We give the \(L^2\) homotopy on the final domains and choose the
level gaps for repeated applications of Lemma~\ref{fi:lem:composition}.
The coordinate chart and its auxiliary polydisc may depend on the
tuple. We choose this polydisc to contain the outer intersection
with a fixed relative margin and decrease every defining bound at
the final restriction.

\subsection{The local restriction identity}\label{fi:sc:contraction}\label{fi:sc:partition}
Let \(R\) be restriction between final polyhedra \(U^-\subset U^+=U\)
with the margins in \eqref{eq:nested-domain-ratios}. Each component
\(W\) of \(U\) is bounded and pseudoconvex. For every closed scalar
\((0,b)\)-form \(f\in L^2(W)\) of positive degree, H\"ormander's estimate
\cite{Hor65}*{Theorem~2.2.3, (2.2.4), p.~107} gives a solution of
\(\db v=f\) with \(b\|v\|^2\le \mathrm e\,(\operatorname{diam}W)^2\|f\|^2\).
Apply this to each coefficient in the chosen holomorphic frame and sum
over the components. The norm comparison \eqref{fi:sc:normequivalence}
then bounds the solution in the bundle norm by \(Cp^M\|f\|\).

Let \(\Pi_b\) be orthogonal projection onto \(\ker\db\) in degree \(b\)
on \(U\). For \(b>0\) let \(\mathsf h_bu\) be the solution of least
norm to \(\db v=\Pi_bu\), and put \(\mathsf h_0=0\). These operators
are linear and polynomially bounded from \(L^2\) to the maximal
domain with its graph norm. For \(u\) in that domain,
\(\mathsf h_{b+1}\db u=u-\Pi_bu\), since the right-hand side is the
unique solution orthogonal to \(\ker\db\)
\cite{Hor65}*{proof of Theorem~1.1.1}. Hence, with
\begin{equation}\label{fi:sc:perturbation}
 T_b=R\mathsf h_b,\qquad K_0=R\Pi_0,\qquad K_b=0\quad(b>0),
\end{equation}
restriction commutes with \(\db\) and gives
\begin{equation}\label{fi:sc:local-homotopy}
 R-K_b=\db T_b+T_{b+1}\db .
\end{equation}
Both terms on the right extend boundedly to \(L^2\): they are
\(R\Pi_b\) and \(R(I-\Pi_b)\) in positive degree, and the first is
zero in degree zero. The coefficients of \(\Pi_0u\) in the chosen
frame are holomorphic. Thus \(K_0\) factors through coefficientwise
restriction \(A^2(U)^{\oplus\rk E}\to A^2(U^-)^{\oplus\rk E}\).
The maps taking and reconstructing the coefficients have polynomial
norm bounds by \eqref{fi:sc:normequivalence}.

\subsection{Choice of degrees and nested domains}\label{app:domain-levels}\label{fi:new:par:nested-levels}

Fix \(k\) and \(r\). There are at most \(H=2k+3\) covers and each
parameter copy has at most \(s_*=n+2(k+1)r\) coordinates. All covers
use the same degree bounds,
\begin{equation}\label{fi:new:eq:degree-bounds}
 0\le q\le n\ \text{for identities},\qquad
 0\le q\le n+1\ \text{for homotopies},\qquad J=n+1 .
\end{equation}
Each tuple has at most \(J+1\) copies, so the total number of
coordinates is at most \(n+H(J+1)s_*\). A parameter copy contributes
at most \(s_*\) coordinate inequalities and \(s_*\) tube inequalities.
An ordinary subdivision uses at most \((J+1)s_*\) additional
inequalities per tuple. Thus all dimensions and numbers of inequalities
are bounded in terms of \(n,k\).

\paragraph{Nested domains.}
Let \([s_0-\beta,s_0]\), \(\beta>0\), be the interval available for one
cover. Take
\[
 \lambda_j=s_0-(j+1)\Delta,\qquad
 \Delta=\frac{\beta}{2(J+2)},\qquad 0\le j\le J+1 .
\]
The innermost tuple level is \(\lambda_{J+1}=s_0-\beta/2\).
Choose the cutoffs on the full outer source, supported in the cover
members at \(s_0-3\beta/4\), and take \(s_0-\beta\) for the smaller
source. A tuple of degree \(a\) passes its interval
\([\lambda_{a+1},\lambda_a]\) to the next cover. Starting from width
\(\beta_0\), at most \(H\) covers give gaps at least
\(\beta_0[2(J+2)]^{-H}\), independently of \(p\).

\section{From complex cells to products}\label{app:cells}
The graph charts in Proposition~\ref{fi:cell:prop:targetproducts}
are defined on complex cells. To obtain product domains, we retain
the disc coordinates and replace each annular or punctured-disc
coordinate by a holomorphic section. We first bound the algebraic
degrees of these substitutions, then choose the cell extension so
that their total error is less than \((40p)^{-1}\).

\subsection{Degree bounds for algebraic functions}
We measure algebraic complexity by graph degree as in
Section~\ref{fi:cell:sec:cells} and \cite{BN19}*{\S2.2.3}.
For a scalar algebraic holomorphic function \(f\), conjugating its
graph equations gives the same complexity for
\(z\mapsto\overline{f(\bar z)}\) on the conjugate domain.

\begin{lemma}[Preservation of algebraic complexity]
\label{fi:it:lem:complexity}
In fixed dimensions, bounded algebraic complexity is preserved under
a fixed number of sums, products, holomorphic quotients, compositions,
evaluations at fixed parameter values, and derivatives. The resulting
bound is independent of the scalar constants in these operations.
\end{lemma}
\begin{proof}
For a composition, the graph degrees and dimensions of the maps bound
the degree of the composite \cite{BN19}*{\S2.2.3, (31)}. Apply this
to addition, multiplication and division to obtain the bounds for
sums, products and quotients on their domains of definition. If a
quotient extends holomorphically across the zero set of its
denominator, its polynomial relation persists there by continuity.

For specialization at \(s_0\), restrict a polynomial relation for
\(f(z,s)\) to a generic line \(s=s_0+\tau v\). Divide the coefficients
by their largest common power of \(\tau\), then put \(\tau=0\).
This leaves a nonzero relation for \(f(z,s_0)\) of no higher degree.

For a derivative, apply the same specialization at \(s=0\) to
\[
 \frac{f(z+se_j)-f(z)}{s},
\]
where \(e_j\) is the \(j\)-th coordinate vector. The quotient has
bounded degree for \(s\ne0\) and extends holomorphically across
\(s=0\) with value \(\partial_{z_j}f\). Iterate for higher derivatives.
\end{proof}

\subsection{Cellular parametrization}

Over a point \(b\) of a base cell, let \(z\in\mathbb C\) be the next
coordinate. For nonvanishing bounded holomorphic functions \(a,g\)
of \(b\), use the fiber notation
\[
 \begin{aligned}
 D(g)&=\{z:|z|<|g(b)|\},&
 D_\circ(g)&=\{z:0<|z|<|g(b)|\},\\
 A(a,g)&=\{z:|a(b)|<|z|<|g(b)|\}.
 \end{aligned}
\]
In the annular case we require \(0<|a(b)|<|g(b)|\). A complex cell is
built one coordinate at a time, by adjoining one of these fibers or a
point. An extension enlarges the base and each fiber
\cite{BN19}*{Definitions~7--8}. On the extended base the original
radius functions are holomorphic, the disc radii are nonzero, and
the annular radii satisfy \(0<|a|<|g|\).

Keep all coordinates in their original order, including point
coordinates. The \(j\)-th component of a cellular map depends only on
the first \(j\) coordinates. The complexity of an algebraic cell is
the maximum of the graph degrees of its radius functions and those
of the base \cite{BN19}*{\S2.2.3}.

Fix the target extension in the algebraic case of
\cite{BN19}*{Theorem~8}. For bounded algebraic holomorphic functions
on an extended cell, the theorem gives a cover by cellular maps on
which the pullback of each function is either identically zero or
nowhere zero. With the input complexity, number of functions, cell
length and target extension fixed, the chart count is polynomially
bounded in the reciprocal of the source extension parameter. The
degrees of the maps are bounded independently of this parameter.

\paragraph{Degrees of the source radii.}
For bounded algebraic input, the source cells also have uniformly
bounded degree \cite{BNICM2023}*{\S3.2, p.~1448, after Conjecture~21}.
Indeed, the intermediate radii are constant, are pulled back from earlier
cells, or are formed from earlier radii and differences of algebraic
branches on covers of bounded degree
\cite{BN19}*{\S\S6.3.1, 7.1, 8.1, (119), (145), (150)}.
Lemma~\ref{fi:it:lem:complexity} bounds their degrees. Refinement to
the chosen source extension preserves this bound for both cells and
maps \cite{BN19}*{Theorem~9 and \S8.1.2}.

\subsection{Product charts on an extended cell}\label{app:wide-cells}

For \(\rho>0\), let \(C^{\{\rho\}}\) denote the extension of \(C\)
in \cite{BN19}*{\S2.1, equation~(29)}. For a fiber of type
\(F\in\{D,D_\circ,A\}\), write \(\delta_F=\delta_F(\rho)\in(0,1)\)
for the ratio that enlarges the outer radius \(|g|\) to
\(|g|/\delta_F\). For an annulus the inner radius \(|a|\) becomes
\(\delta_F|a|\). These ratios satisfy
\[
 \rho=\frac{2\pi\delta_D}{1-\delta_D^2},
 \qquad
 \delta_{D_\circ}(\rho)=\delta_A(\rho)
       =\exp\!\left(-\frac{\pi^2}{2\rho}\right).
\]

Put \(\epsilon=(40p)^{-1}\) as in
Proposition~\ref{fi:cell:prop:targetproducts}, and let \(s=N+m\)
be the cell length. The \(z\)-projection of the graph cover lies in
a fixed polydisc compactly contained in \(\Delta\). Let \(M_0\ge1\)
bound the Euclidean norm of this projection. Apply
\cite{BN19}*{Theorem~8} with the refinement of the source extension
in \S8.1.2 by Theorem~9, taking
\[
 \frac1{\sigma}=\max\left\{1,\,
          \frac2{\pi^2}\log\frac{8(s+1)M_0}{\epsilon}\right\},
 \qquad
 \delta=\exp\!\left(-\frac{\pi^2}{2\sigma}\right).
\]
Let \(F\) be the \(z\)-projection of one chart, so that
\(\|F\|\le M_0\) on its \(\{\sigma\}\)-extension. Work through the
fibers in reverse order. At each stage the remaining variables are
the extended base coordinates and the disc parameters already retained.
A point fiber adds no parameter. For a disc \(D(g(b))\), write
\(z=g(b)w\) and retain \(w\). The original disc is then \(|w|<1\).
Since \(\sigma\le1\), \cite{BN19}*{(29)} shows that the extended disc
contains \(|w|<2\).

For an annulus \(A(a(b),g(b))\), put \(\gamma=a(b)/g(b)\) and
\(z=g(b)\zeta\). By \cite{BN19}*{Definition~8} we have
\(0<|\gamma|<1\) on the extended base. Fix \(b\) and the retained
disc parameters. The map has norm at most \(M_0\) on
\(\delta|\gamma|<|\zeta|<\delta^{-1}\). Apply Cauchy's estimates to
its Laurent coefficients \(c_j\) on concentric circles and let the
radii tend to the two boundary radii. This gives
\[
 |c_j|\le M_0\delta^j,\qquad
 |c_{-j}|\le M_0(\delta|\gamma|)^j\qquad(j\ge1).
\]
On \(|\gamma|<|\zeta|\le1\), the positive and negative powers each
contribute at most \(M_0\delta/(1-\delta)\) in sum of norms. Compare
the value at \(\zeta\) with the value at \(1\). The constant terms
cancel, so the section \(z=g(b)\) changes the value by at most
\(4M_0\delta/(1-\delta)\). On a punctured disc, boundedness gives a
holomorphic extension across zero. Only the nonnegative powers occur,
so the same comparison gives \(2M_0\delta/(1-\delta)\).

Each substitution maps the extended base and the retained discs into
the preceding extended domain, so \(M_0\) bounds the map at every
stage. For a given point of the original cell, take \(w=z/g(b)\) at
each disc fiber. Since we remove the fibers in reverse order, each
comparison uses the same base point for the two values. There are at
most \(s\) substitutions. With \(\delta\le\epsilon/[8(s+1)M_0]\)
and \(\delta<1/2\), the total error is at most
\(4sM_0\delta/(1-\delta)<\epsilon\).

The resulting map is holomorphic on the product \(|w_j|<2\), and
the image of \(|w_j|<1\) gives the required approximation. Each of
the at most \(s\) substitutions uses a radius function of the preceding
coordinates. Lemma~\ref{fi:it:lem:complexity} therefore bounds the
degree of the map. Since the substitutions respect the cellular order,
\(Q\circ F\) depends only on the parameters retained from the first
\(m\) fibers.

\section*{Acknowledgments}

The author is grateful to Professor Fusheng Deng and Professor Xiaonan Ma for their valuable
suggestions on related topics.

\section*{AI Statement}

The author conceived and developed the core ideas and proof strategy of this paper. ChatGPT was used as an aid in exploring possible constructions and checking the proofs. The author has independently verified the entire manuscript, including its mathematical arguments, references, and conclusions, and takes full responsibility for its contents.

\end{document}